\documentclass[12pt]{article}
\usepackage{amssymb}
\usepackage{amsmath}
\usepackage{amsfonts}
\usepackage{amsthm}
\usepackage{color}

\def\pd#1#2{\frac{\partial#1}{\partial#2}}
\def\di{\bigstar}

\newcommand{\bea}{\begin{eqnarray}}
\newcommand{\eea}{\end{eqnarray}}

\newcommand{\be}{\begin{equation}}
\newcommand{\ee}{\end{equation}}

\theoremstyle{plain}
\newtheorem{theorem}{Theorem}
\newtheorem{corollary}[theorem]{Corollary}
\newtheorem{proposition}[theorem]{Proposition}
\newtheorem{lemma}[theorem]{Lemma}

\theoremstyle{definition}
\newtheorem{definition}[theorem]{Definition}
\newtheorem{note}[theorem]{Note}
\newtheorem{example}{Example}
\newtheorem{remark}{Remark}

\def\R{\mathbb{R}}
\def\Diff{\operatorname{Diff}}

\def\cG{\mathcal{G}}
\def\bX{\mathbf{X}}
\def\bY{\mathbf{Y}}

\newcommand{\beas}{\begin{eqnarray*}}
\newcommand{\eeas}{\end{eqnarray*}}

\numberwithin{equation}{section}
\numberwithin{theorem}{section}
\numberwithin{example}{section}

\begin{document}

\centerline{\Large {\bf Quasi-Lie schemes for PDEs}}
\vskip 0.5cm

\centerline{ Jos\'e F. Cari\~nena$^{(a)}$, Janusz Grabowski$^{(b)}$
 and Javier de Lucas$^{(c)}$}
\vskip 0.5cm

\centerline{$^{(a)}$Departamento de  F\'{\i}sica Te\'orica, Universidad de Zaragoza,}
\medskip
\centerline{P. Cerbuna 12, 50009 Zaragoza, Spain}
\medskip
\centerline{$^{(b)}$Institute of Mathematics, Polish Academy of Sciences,}
\medskip
\centerline{ul. \'Sniadeckich 8, 00-656 Warszawa, Poland}
\medskip
\centerline{$^{(c)}$Department of Mathematical Methods in Physics, University of Warsaw,}
\medskip
\centerline{ul. Pasteura 5, 02-093 Warszawa, Poland}
\medskip

\vskip 1cm

\begin{abstract}{
The theory of \emph{quasi-Lie systems},
i.e. systems of {first order} ordinary differential equations which can be related \emph{via} a generalised flow to Lie systems, is extended to systems of partial differential equations and its applications to obtaining $t$-dependent superposition rules and integrability conditions are analysed. We develop a procedure of constructing quasi-Lie systems through a generalisation to PDEs of the so-called theory of \emph{quasi-Lie schemes}. Our techniques are illustrated with the analysis of Wess-Zumino-Novikov-Witten models, generalised Abel differential equations, B\"acklund transformations, as well as other differential equations of physical and mathematical relevance.
}




\end{abstract}

\section{Introduction.}
A {\it Lie system} is a non-autonomous system of
first-order ordinary differential equations whose general solution can be written as an autonomous function, a so-called {\it superposition rule}, depending on a generic finite set of particular solutions and some constants \cite{CGM07,LS}. Examples of Lie systems are matrix Riccati equations \cite{PW}, non-autonomous linear systems of first-order ordinary differential equations \cite{Dissertationes}, and types of Bernoulli equations (cf. \cite{BHLS15,Dissertationes}).

Superposition rules are explicitly known even for systems of differential equations whose general solutions are not,
like in the case of Riccati equations and most of their generalisations \cite{CLR07b,In44}. As a consequence, superposition rules simplify the application of numerical methods \cite{PW}, and they enable us to analyse the properties of Lie systems, e.g. their periodic orbits \cite{LTS16}, without knowing their general solutions.

Since the foundational works by Vessiot and Lie \cite{LS,Ve93,Ve99}, the theory of Lie systems has described many of their geometric properties and physical applications \cite{CGM00,CGM07,Dissertationes,LS,Ve93,Ve99,PW}. For instance, the obtention of the general solution of a Lie system can be reduced to integrating a particular type of Lie systems on a Lie group \cite{CGR01,Ve93}, and many applications ranging from Smorodinsky-Winternitz oscillators to matrix Riccati equations were studied during the 80's by Winternitz and his collaborators  \cite{HWA83,BecHusWin86,BecHusWin86b,OlmRodWin86,OlmRodWin87,PW}. Moreover, geometric methods \cite{BCHLS13,LV15} have been developed to derive superposition rules without integrating complicated systems of partial and/or ordinary differential equations as needed in standard methods \cite{Dissertationes,LS,PW}. Additionally, many works have addressed the study of autonomous and non-autonomous constants of the motion, exact solutions, integrability conditions, and other interesting
properties for particular Lie systems with relevant physical and mathematical applications \cite{HWA83, BecHusWin86,BecHusWin86b,CL08M,CLR08c,Ru10,OlmRodWin86, OlmRodWin87}.

Lie succeeded in characterising non-autonomous systems of first-order ordinary differential equations admitting a superposition rule \cite{LS}. His result,  the nowadays known
as {\it Lie--Scheffers Theorem} \cite{CGM00}, can be expressed in modern geometric terms by recalling that a non-autonomous system of first-order ordinary differential equations in normal form amounts to a time-dependent vector field $X$ \cite{Dissertationes}. In fact, the Lie--Scheffers Theorem states that a system $X$ is
a Lie system if and only if $X$ takes values in a finite-dimensional Lie algebra of vector fields, a so-called {\it Vessiot--Guldberg Lie algebra} of the Lie system \cite{CS16,Dissertationes,IG17}.

The fact that every Lie system is related to a Vessiot--Guldberg Lie algebra implies that being a Lie system is rather exceptional \cite{CGL09,In65}. To extend the theory of Lie systems to a broader family of differential equations, some generalisations of this theory have been proposed. Such generalisations are aimed at investigating a kind of partial differential
equations (the hereafter referred to as {\it PDE Lie systems} \cite{CGM07,OG00}), classes of second- and higher-order ordinary differential equations (the so-called {\it HODE Lie systems} \cite{CLR08c}), types of Schr\"{o}dinger equations \cite{CarRam05b}, etcetera \cite{BCFH17}.

Among above-mentioned theories, quasi-Lie schemes \cite{CGL08,CLL08Emd,CL08Diss} represent a geometric structure aimed at the construction of a simple non-autonomous dependent change of variables mapping a class of non-autonomous systems of first-order ordinary differential equations, the {\it quasi-Lie systems}, into Lie systems. Quasi-Lie systems admit non-autonomous superposition rules, whose properties were analysed in the theory of {\it Lie families} in \cite{CGL09}. Moreover, quasi-Lie schemes have also found applications to the description of integrability conditions for Abel differential equations and other differential equations \cite{CLL08Emd,CL08Diss,CL15}.

The main goal of this work is to extend the theory of quasi-Lie schemes and quasi-Lie systems to study non-autonomous superposition rules and integrability conditions for families of systems of partial differential equations (PDEs).
We will deal mainly with integrable systems of first-order PDEs in $s$ independent coordinates $t=(t_1,\dots,t_s)$  and $n$ dependent coordinates $x=(x^1,\dots,x^n)$ on a manifold $N$, let us say
\begin{equation}\label{PDE0}
\frac{\partial x^i}{\partial t_\pi}=X^i_\pi(t,x),\qquad \pi=1,\dots,s,\quad i=1,\dots,n\,.
\end{equation}
In this paper, `an integrable system of PDEs' is one such that the {\it zero curvature condition} (ZCC) is satisfied \cite{Da89}. This ensures that (\ref{PDE0}) admits a particular solution for every initial condition, i.e. there exists a unique solution $x(t)$ to (\ref{PDE0}) satisfying $x(t_0)=x_0$ for every $x_0\in N$ (cf. \cite{FT87,ZS79}).

Although most systems of PDEs in the literature are not of the type considered in this work (cf. \cite{PZ04,PZM02}), relevant differential equations can be brought into the form (\ref{PDE0}), e.g. linear spectral problems and soliton surfaces in Lie algebras \cite{GG11,Gr16}, the Von Misses transformation for studying Navier-Stockes equations \cite{VM27,DZ98}, Toda lattices \cite{FGRSZ99}, B\"acklund transformations to analyse heat equations and modified KdV equations through Burguers equations and KdV equations respectively \cite{PZ04}, and others \cite{CZ14}.

A general and standard method for finding particular solutions or at least relevant properties for systems of PDEs (\ref{PDE0}) depends on determining a suitable \emph{bundle change of coordinates} in the bundle $\textrm{pr}_1:\mathbb{R}^s\times N\to \mathbb{R}^s$, i.e. a local diffeomorphism on $\mathbb{R}^s\times N$ of the form $\phi(t,x)=(t,y(t,x))$ such that the system (\ref{PDE0}) takes in the coordinates $(t,y)$ a more appropriate form. In this case  we say that the initial and the transformed systems are bundle $\phi$-related.
Although all systems (\ref{PDE0}) with the same values of $n$ and $s$ are locally bundle $\phi$-related (see Proposition \ref{p1}), the determination of a diffeomorphism $\phi$ to map the system (\ref{PDE0}) onto, for instance, a zero system, requires to know the general solution of (\ref{PDE0}) explicitly. This may cause one to think that the method is useless. But it is fortunately not the case, as we can look for transformations which can still be computed explicitly and yield a relation of our system to much nicer ones.

In a nutshell, we find methods to map a certain family of systems of PDEs of the form (\ref{PDE0}) onto a simpler family of bundle $\phi$-related systems of PDEs by our extension of quasi-Lie schemes to systems of PDEs.  If the transformed systems of PDEs form a family of PDE Lie systems related to a common Vessiot--Guldberg Lie algebra, then our approach ensures the existence of the so-called {\it common $t$-dependent superposition rules} for the initial family of systems of PDEs, namely a $t$-dependent function expressing the general solution of any initial system of PDEs of the form (\ref{PDE0}) in terms of a generic finite family of its particular solutions and some parameters. The initial family of PDEs then becomes a hereafter called {\it PDE Lie family}, i.e. a family of systems of first-order PDEs admitting a common $t$-dependent superposition rule.

In order to prove when a certain family of systems of PDEs can be related to a PDE Lie system with a fixed Vessiot--Guldberg Lie algebra, the hereafter called generalised PDE Lie--Scheffers Theorem (Theorem \ref{MT}) is proved. This theorem is a natural generalisation of the Lie--Scheffers Theorem characterising Lie systems \cite{CGM07,LS} that enables one to characterise {PDE Lie families}. New examples and applications of PDE Lie systems are given: reductions of Wess-Zumino-Novikov-Witten models, PDE Lie systems whose integrability conditions describe sine-Gordon equations, and others. This is specially relevant due to the lack of applications of these systems in the literature, which are rather studied only theoretically \cite{CGLS14,Dissertationes,GR95,OG00}.




Finally, the theory of quasi-Lie invariants  \cite{CL15} is extended to systems of PDEs. Roughly speaking, every quasi-Lie scheme for systems of PDEs provides a naturally family of $t$-dependent transformations mapping elements of a family of systems of PDEs into new members of the same family. Quasi-Lie invariants are certain functions taking the same values on systems of PDEs related by the $t$-dependent transformations of a quasi-Lie scheme. This provides clues to know whether a certain system of PDEs can be mapped onto PDE Lie systems and other simpler systems of PDEs. It also simplifies previous techniques to obtain quasi-Lie invariants. Our methods have applications in the theory of integrability of systems of ordinary and partial differential equations, as quasi-Lie invariants can be seen as integrability conditions for systems of PDEs. A simple example illustrating the use of quasi-Lie invariants in the integrability problem of generalised Abel differential equations \cite{CR00,TR05,Ch31,Ch40}  by the so-called  {\it generalised Chiellini conditions} {\cite{HLM13,HLM14}} is developed. In this case, we prove that generalised Chiellini conditions are indeed quasi-Lie invariants for generalised Abel differential equations.

The organisation of the paper goes as follows. Section 2 describes generalised $t$-flows and $t$-dependent polyvector fields. In Section 3 we apply bundle transformations to relate systems of PDEs.  Basics about PDE Lie systems and several new applications are contained in Section 4. Section 5 discusses $t$-dependent superposition rules for families of systems of PDEs. The
generalised PDE Lie--Scheffers Theorem for common $t$-dependent superposition rules is proved in Section 6.
In Section 7 we study quasi-Lie schemes for systems of PDEs. Section 8 provides several clues in the use of quasi-Lie schemes to find integrability conditions for systems of PDEs. We detail our conclusions and further work in Section 9.

\section{Generalised flows and $t$-dependent polyvector fields}
This section presents a generalisation of the results on generalised flows and time-dependent vector fields for systems of first-order ordinary differential equations given in \cite{CGL08} to the realm of systems of PDEs. To highlight the key points of our work, we will hereafter assume that structures are smooth and globally defined. Technical details follow straightforwardly by generalising the ideas of \cite{CGL08}.

Let $\{e^1,\ldots,e^s\}$ be a basis of $\mathbb{R}^s$. Consider local coordinate systems $\{x^1,\ldots, x^n\}$ on a manifold $N$ and $\{t_1,\ldots, t_s\}$ on $\mathbb{R}^s$. We hereafter write $t=(t_1,\ldots,t_s)$ and we call $t$ time when $s=1$. The system of PDEs (\ref{PDE0}) can be represented geometrically
by the $t$-dependent polyvector field ${\bf X}:\mathbb{R}^s\times N\rightarrow TN\otimes \mathbb{R}^s$ on $N$ given by
$$
{\bf X}(t,x)=\sum_{\pi=1}^s\sum_{i=1}^nX^i_\pi(t,x)\frac{\partial}{\partial x^i}\otimes e^k\,,
$$
or equivalently as an Ehresmann connection $\nabla$ in the trivial fibration
$\tau:\R^s\times N\rightarrow\R^s\,,$
whose horizontal distribution, $H(\nabla)$, is spanned by the vector fields
$$\bar X^{[k]}=\partial_{t_\pi}+X_k=\partial_{t_\pi}+\sum_{i=1}^nX^i_\pi\partial_{x^i}\,,
\quad \pi=1,\dots,s\,,
$$
where $\partial_x$ stands for $\partial/\partial x$ and the $X^i_\pi$, with $i=1,\ldots,n$ and $\pi=1,\ldots,s$,  are functions on $\mathbb{R}^s\times N$.
Another useful realisation is to view $\bX$ as a vector $(X_1,\dots,X_s)$ whose entries are
{\it $t$-dependent vector fields on $N$}, namely mappings $X_\pi:\mathbb{R}\times N\rightarrow TN$ such that every $X_i(t,\cdot):N\rightarrow TN$ is a standard vector field on $N$ for every $t\in\mathbb{R}^s$.

By the \emph{autonomisation} of ${\bf X}$, we understand the polyvector field $\bar\bX=(\bar X^{[1]},\dots,\bar X^{[s]})$ on $\R^s\times N$. It is also {worth} noting that every $t$-dependent polyvector field ${\bf X}$ on $N$ amounts to a $t$-parametrised family of polyvector fields ${\bf X}_t:x\in N\mapsto {\bf X}(t,x)\in TN\otimes \mathbb{R}^s$. We say that ${\bf X}=(X_1,\ldots,X_s)$ takes values in a Lie algebra of vector fields if all the vector fields $(X_1)_t,\ldots,(X_s)_t$, with $t\in\mathbb{R}$, do so.


A particular solution to the system of PDEs (\ref{PDE0}) is given by a map $t\in \mathbb{R}^s\mapsto \gamma(t)\in N$ such that
\begin{equation}\label{e1}
\frac{\partial\gamma}{\partial t_\pi}(t)= X_\pi(t, \gamma(t)),\qquad \pi= 1,\dots,s .
\end{equation}
We hereafter assume that system (\ref{e1}) is {\it integrable}, namely the components of the $t$-dependent polyvector field ${\bf X}$ satisfy the so-called {\it zero curvature condition} (ZCC), namely \cite{Da89}
\begin{equation}\label{IntCon}
[\bar X^{[\pi]},\bar X^{[\nu]}]=0\,,\qquad \pi,\nu=1,\ldots,s.
\end{equation}
This amounts to the fact that the connection $\nabla$ is flat, i.e. the horizontal distribution $H(\nabla)$ is
integrable giving rise to a foliation $\mathfrak{F}(\nabla)$ on $\mathbb{R}^s\times N$. The leaves of $\mathfrak{F}(\nabla)$ amount to particular solutions of (\ref{PDE0}). When the system of PDEs (\ref{PDE0}) associated with the $t$-dependent polyvector field ${\bf X}$ is integrable, we say that ${\bf X}$ is {\it intregrable}.

The ZCC condition (\ref{IntCon}) ensures that there exists for each $x_0\in N$ and $t_0\in \mathbb{R}^s$ a unique maximal
solution $\gamma_\bX^{x_0}(t)$  of system (\ref{e1}) with the initial value $x_0$ at $t_0$, e.g. satisfying
$\gamma_\bX^{x_0}(t_0)=x_0$. 
Each solution of (\ref{e1}) is called an {\it integral submanifold} or simply a {\it particular solution} of ${\bf X}$.

A typical method of finding particular solutions and/or their properties relies on determining an appropriate \emph{bundle isomorphism}  {in the bundle $\textrm{pr}_1:\mathbb{R}^s\times N\to \mathbb{R}^s$}, i.e. a local diffeomorphism {$\phi: (t,x)\in  \mathbb{R}^s\times N\to (t,y(t,x))\in \mathbb{R}^s\times N$,} such that $\phi$ transforms the system of PDEs (\ref{PDE0}) into  a new system ${\bf X}'$ of PDEs taking a nicer form we can deal with. Equivalently, $\phi$ can be understood as a bundle change of variables mapping the system of PDEs (\ref{PDE0}) into a more appropriate form. If $\phi$ maps ${\bf X}$ onto ${\bf X}'$, then we will say that the systems ${\bf X}$ and ${\bf X}'$ are {\it bundle $\phi$-related}.  However, one has the following discouraging proposition.

\begin{proposition}\label{p1} Any two integrable systems of PDEs related to $t$-dependent polyvector fields ${\bf X}_1$ and ${\bf X}_2$ are bundle $\phi$-related  for some bundle isomorphism $\phi:\mathbb{R}^s\times N\rightarrow \mathbb{R}^s\times N$.
\end{proposition}
\begin{proof}
We will show first that the polyvector field ${\bf X}_0=0$ on $N$, whose corresponding connection $\nabla_0$ is generated by the $\partial_{t^\pi}$ with $\pi=1,\dots,s$, can always be mapped onto an arbitrary system of PDEs ${\bf X}$ of the form (\ref{PDE0}). Since ${\bf X}$ satisfies the ZCC condition by assumption, there exists  a particular solution $t\mapsto y(t,x)=\gamma_\bX^{x}(t)$  of ${\bf X}$, i.e. a leaf of $\mathfrak{F}(\nabla)$, for each boundary condition  $\gamma_\bX^{x}(0)=x\in N$. The map $\phi(t,x)=(t,y(t,x))$ maps the leaf of $\mathfrak{F}(\nabla_0)$  going through $(0,x)$, i.e. the horizontal plane $\mathbb{R}^s\times \{x\}$, onto the unique leaf of $\mathfrak{F}(\nabla)$ going through $(0,x)$. This induces a bundle isomorphism $\phi:\mathbb{R}^s\times N\rightarrow \mathbb{R}^s\times N$ mapping the particular solutions of ${\bf X}_0$ onto ${\bf X}$ and vice versa.

In view of the above, ${\bf X}_0$ can be mapped onto ${\bf X}_1$ and ${\bf X}_2$ through two $\phi$-bundle isomorphisms $\phi_1,\phi_2$ respectively, and then $\phi_2\circ \phi^{-1}_1$ maps ${\bf X}_1$ onto ${\bf X}_2$.
\end{proof}

Although the situation is very simple geometrically, the explicit description of the bundle isomorphism $\phi$ in Proposition \ref{p1} requires to know all particular solutions of (\ref{PDE0}). One can therefore think that the method is useless. This is fortunately not true, as we can look for transformations that can still be computed explicitly and yield a relation of our system to a nicer one (although non-zero).
\begin{example}\label{e1a}
Let us illustrate how bundle $\phi$-related systems appear in an integration problem for the so-called generalised Abel differential equations of the first kind \cite{Al07,Bo05,TR05,Ch40,SP03,SR82,ZT09}. In particular, consider the generalised Abel differential equation of the form
\begin{equation}\label{InAbe}
\frac{dx}{dt}=a(t)+c(t)x+f(t)x^{\epsilon-1}+g(t)x^\epsilon,\qquad x\in \mathbb{R},\quad \epsilon\in\mathbb{R}\backslash\{1\},\quad t\in \mathbb{R},
\end{equation}
for some $t$-dependent functions $a(t),c(t),f(t),g(t)$ satisfying that $f(t),g(t)\neq 0$ and the so-called {\it generalised Chiellini's condition} (see \cite[Theorem 1]{HLM13} and  \cite[Lemma 2]{HLM14})
\begin{equation}\label{GCC}
\frac{g^{\epsilon-1}(t)}{f^\epsilon(t)}e^{\int c(t)dt}\frac{d}{dt}\left[\frac{f(t)e^{-\int c(t)dt}}{g(t)}\right]=-k_1\in\mathbb{R},\qquad \frac{g^{\epsilon-1}(t)}{f^{\epsilon}(t)}a(t)=k_2\in\mathbb{R}.
\end{equation}
If $c(t)=a(t)=0$, then above conditions reduce to the standard {\it Chiellini's condition} \cite{Ch31,Ch40,MR}.
Under the assumptions (\ref{GCC}), the $t$-dependent bundle change of variables
\begin{equation}\label{Change}
(t,y(t))=\left(t,\frac{f(t)}{g(t)}x(t)\right),
\end{equation}
i.e.
the bundle isomorphism
$$
\phi:(t,x)\in \mathbb{R}\times \mathbb{R}\mapsto \left(t,\frac{f(t)}{g(t)}x\right)\in \mathbb{R}\times \mathbb{R}
$$
maps the generalised Abel differential equation (\ref{InAbe}) onto the simpler one
\begin{equation}\label{OutAbe}
\frac{d{y}}{dt}=\frac{a(t)g(t)}{f(t)}\left[1+\frac{1}{k_2}(y^\epsilon+y^{\epsilon-1}+k_1y)\right],
\end{equation}
whose general solution can be obtained by making a $t$-dependent reparametrisation and a quadrature. In other words, our generalised Abel differential equation (\ref{InAbe}) is bundle $\phi$-related to the generalised Abel differential equation (\ref{OutAbe}), whose solution is straightforward. The key of the previous procedure is to find a geometric method to derive the $t$-dependent change of variables (\ref{Change}). This procedure also suggests that this can be done by assuming some relations between the functions $a(t),c(t),f(t),$ and $g(t)$.
\end{example}

A {\it generalised $t$-flow} $g$ on $N$ with foot point $t_0$ is a $t$-dependent family $\{g_t\}_{t\in\mathbb{R}^s}$ of
diffeomorphisms on $N$ such that
$g_{t_0}=\text{id}_N$. Alternatively, $g$ can be understood as a curve in the group of diffeomorphisms $g:t\in
\R^s\mapsto g_t\in\Diff(N)$ with
$g_{t_0}=\text{id}_N$. If not otherwise stated, we will assume that generalised $t$-flows have foot point $t_0$.
\color{black}


It is opportune to {\it autonomise} the generalised $t$-flow $g$ giving rise to a single local
diffeomorphism $\bar{g}(t,x)=(t,g(t,x))$ on $\mathbb{R}^s\times N$. Then, every integrable $t$-dependent polyvector field ${\bf X}$ on $N$ induces a generalised $t$-flow $g^{\bf X}$ defined by
$g_t^{\bX}(x_0)=\gamma_{\bX}^{x_0}(t)\,,$
where $\gamma_{\bX}^{x_0}(t)$ stands for the particular solution of ${\bf X}$ with initial condition $\gamma_{\bX}^{x_0}(t_0)=x_0\in N$.
If we write $g=g^{\bf X}$, then
\begin{equation}\label{e4}
(X_\pi)_t(x):= X_\pi(t,x)=\frac{\partial g_t}{\partial t_\pi}\circ g_t^{-1}(x)\,,\qquad \pi=1,\ldots,s,
\end{equation}
where $ (X_\pi)_t$ and $\partial g_t/\partial t_\pi$ are understood as maps from $N$ into $TN$. The equation (\ref{e4}) defines a one-to-one correspondence between generalised
$t$-flows at a fixed foot point $t_0$ and integrable $t$-dependent polyvector fields on $N$. This result is summarised in the following theorem, which generalises Theorem 1 in \cite{CGL08}.

\begin{theorem}\label{t1} Equation (\ref{e4}) defines a
one-to-one correspondence between the germs of generalised $t$-flows at a fixed foot point and the germs of integrable $t$-dependent polyvector fields
on $N$.
\end{theorem}

Any two generalised $t$-flows $g$ and $h$ with a common  foot point can be composed: by definition $(g\circ h)_t=g_t\circ h_t$. As generalised $t$-flows correspond to integrable $t$-dependent polyvector fields, this gives rise to an action  of a generalised $t$-flow $h$ on an integrable $t$-dependent polyvector field ${\bf X}$, leading to the integrable $t$-dependent polyvector field $h_\di {\bf X}$ defined  by the equation
\begin{equation*}
g^{h_\di {\bf X}}=h\circ g^{\bf X}\,. \label{e5}
\end{equation*}
A short calculation leads to obtain a more explicit form of this action
\begin{equation}\label{e6}
(h_\di {\bf X})_t=\sum_{\pi=1}^s\left(\pd{h_t}{t_\pi}\circ h_t^{-1}+h_{t*}(X_\pi)_t\right)\otimes e^\pi,
\end{equation}
where $h_{t*}$, for each $t\in\mathbb{R}^s$, is the standard action of the diffeomorphism $h_t$ on vector fields. Similarly, (\ref{e6})
can be rewritten as an action of integrable $t$-dependent polyvector fields on integrable $t$-dependent polyvector fields:
\begin{equation}\label{e7} (g^{{\bf Y}}_\di {\bf X})_t={\bf Y}_t+(g_t^{\bf Y})_*({\bf X}_t),
\end{equation}
where $(g_t^{\bf Y})_*({\bf X}_t)=\sum_{\pi=1}^s[(g_t^{\bf Y})_*(X_\pi)_t]\otimes e^\pi$.
The latter defines a group structure in
$t$-dependent polyvector fields relative to the product ${\bf Y}\star {\bf X}=g_\di^{\bf Y}{\bf X}$. 
Since every generalised $t$-flow has inverse $(g^{-1})_t=(g_t)^{-1}$, the
generalised $t$-flows, or better to say, the corresponding germs, form a group and the formula (\ref{e7}) allows
us to compute the $t$-dependent polyvector field ${\bf X}_t^{-1}$ associated with the
inverse. It is the $t$-dependent polyvector field
\begin{equation*}
{\bf X}_t^{-1}=-(g^{\bf X}_t)^{-1}_*({\bf X}_t)\,.
\end{equation*}
For $t$-independent polyvector fields ${\bf X}$, namely ${\bf X}_t={\bf X}_0$ for all $t\in\mathbb{R}$,  we have $(g^{\bf X}_t)_* {\bf X}={\bf X}$ and therefore
${\bf X}^{-1}=-{\bf X}\,.$
Hence, the integral sections of $h_\di {\bf X}$ are of the form $h_t(\gamma(t))$, where $\gamma(t)$
is any integral section of ${\bf X}$. We can summarise our observation as follows.
\begin{theorem}\label{t2} The equation (\ref{e6}) defines a group action
of generalised $t$-flows on integrable $t$-dependent polyvector fields in the sense that
\begin{equation}\label{NA}
(g\circ h)_\di {\bf X}=g_\di(h_\di {\bf X})\,\qquad \textrm{\rm id}_\di{\bf X}={\bf X},
\end{equation}
where ${\rm id}_t={\rm id}_N$ for every $t\in \mathbb{R}^s$.
The integral sections of $h_\di {\bf X}$ are of the form $h_t(\gamma(t))$, where $\gamma(t)$ is an arbitrary integral
section for ${\bf X}$.
\end{theorem}
The action of generalised $t$-flows on integrable $t$-dependent polyvector fields given by  (\ref{NA}) can be also defined elegantly via autonomisations. This gives rise to the Theorem \ref{t3}, whose proof is straightforward.


\begin{theorem}\label{t3} For any generalised $t$-flow $h$ and any
integrable $t$-dependent polyvector field ${\bf X}$ on $N$, the standard action  of the
diffeomorphism $\bar{h}$  on the $t$-dependent polyvector field $\bar{\bf X}$ is the autonomisation of the $t$-dependent polyvector field $h_\di {\bf X}$, i.e.
$$
\bar{h}_*(\bar{\bf X})=\overline{h_\di {\bf X}}\,.
$$
\end{theorem}

The composition of two generalised $t$-flows $g$ and $h$ with different foot points can be defined as in the case of generalised $t$-flows with the same foot point. Although the result is a well-defined family of diffeomorphisms $\{(g\circ h )_t\}_{t\in\mathbb{R}^s}$, this family may not be a generalised $t$-flow since it does not need to contain ${\rm Id}_N$. Anyway, such families of diffeomorphisms, called hereafter {\it extended $t$-flows}, admit a group structure relative to their composition and, as they are compositions of an arbitrary number of generalised $t$-flows, they also act on integrable $t$-dependent polyvector fields. Moreover, extended $t$-flows can be autonomised as generalised $t$-flows. In this way, the generalisations of Theorems \ref{t2} and \ref{t3} to extended $t$-flows are immediate.
\section{Bundle transformations and integrability conditions}

Numerous works in the literature tackle the description of integrability conditions and/or methods of integration for non-autonomous systems of first-order ordinary and partial differential equations \cite{CR00,HLM13,DZ98}. Although many of them are based upon {\it ad-hoc} techniques, they generally rely on the same simple geometric procedure. This justifies to introduce the hereafter called {\it bundle transformations}, namely bundle {automorphisms} of the bundle $\R^s\times N\to\R^s$, which help in obtaining quasi-Lie systems and integrating the initial differential equation.

For instance, consider the {\it Abel differential equations of the first-kind} \cite{CR00}
\begin{equation}\label{AbeCh}
\frac{dx}{dt}=f_1(t)x^2+f_2(t)x^3,\qquad x\in \mathbb{R},\quad t\in \mathbb{R},
\end{equation}
where $f_1(t)$ and $f_2(t)$ are any  $t$-dependent functions such that $f_1(t),f_2(t)$ are not vanishing. If the celebrated Chiellini condition {is satisfied  \cite{Ch31,MR}}, namely
$$
\frac{d}{dt}\left(\frac{f_2}{f_1}\right)=kf_1,\qquad k\in \mathbb{R},
$$
then the $t$-dependent change of variables $z(t)={f_2(t)}x/{f_1(t)}$
maps the initial Abel differential equation into
\begin{equation}\label{NiceAbel}
\frac{dz}{dt}=\frac{f_1^2(t)}{f_2(t)}(z^2+z^3+kz),\qquad z\in \mathbb{R},
\end{equation}
which is autonomous up to a $t$-dependent reparametrisation and whose general solution can be easily  obtained \cite{In44}. In other words, the initial Abel differential equation can be solved from an easily integrable differential equation by a $t$-dependent change of variables.

The other way round, the above result can be explained by saying that there exists a family of exactly integrable $t$-dependent vector fields spanning a one-dimensional Lie algebra of vector fields on $\mathbb{R}$, namely
$$
Z_{(k)}=\xi(t)(kz+z^2+z^3)\frac{\partial}{\partial z},\qquad k\in \mathbb{R},\qquad \xi(t)\in C^\infty(\mathbb{R}),
$$
that can be transformed via a family of $t$-dependent transformations $$\bar g:(t,x)\in \mathbb{R}\times \mathbb{R}\mapsto (t,\sigma(t) x)\in \R\times\mathbb{R}\,,
$$
where $\sigma(t)$ is a non-vanishing $t$-dependent function,
into a new family of $t$-dependent vector fields
$$
g_\di Z_{(k)}=\dot gg^{-1}+g_{t*}Z_{(k)}=\left[\left(\frac{1}{\sigma(t)}\frac{d\sigma}{dt}(t)+\xi(t)k\right)z+\frac{\xi(t)}
{\sigma(t)}\left(z^2+\frac{1}{\sigma(t)}z^3\right)\right]\frac{\partial}{\partial z}
$$
containing any Abel differential equation satisfying the Chiellini condition. If the initial Abel differential equation (\ref{AbeCh}) is of the above form for a non-vanishing $t$-dependent function $\sigma(t)$ and an arbitrary $t$-dependent function $\xi(t)$, then it can be exactly solved. This occurs if and only if the Chiellini condition is satisfied. Indeed, this ensures that
$$
(f_1(t)z^2+f_2(t)z^3)\frac{\partial}{\partial z}-g_{t*}Z_{(k)}=\frac{1}{\sigma(t)}\frac{d\sigma}{dt}(t)z\frac{\partial}{\partial z},
$$
belongs, for every $t\in \mathbb{R}$, to the set of fundamental vector fields of the action of the commutative
group $\R_*=\R\setminus\{ 0\}$ on the manifold $\R$:
$$\R_*\times\R\ni(\sigma,x)\stackrel{\varkappa}{\mapsto}
 \sigma x\in\R\,.$$ In such a case, the $t$-dependent coefficient $\sigma(t)$ satisfies a differential equation determined by a $t$-dependent vector field taking values in a one-dimensional Lie algebra given by
$$
\dot gg^{-1}=(f_1(t)z^2+f_2(t)z^3)\frac{\partial}{\partial z}-g_{t*}Z_{(k)}=-k\xi(t)z\frac{\partial}{\partial z}.
$$
Note that the map $\phi$ relating the equations (\ref{AbeCh}) and (\ref{NiceAbel}) amounts to a $t$-dependent change of variables $z(t)=\varkappa(\sigma(t),x)$.
As actions induced by the integration of commutative or, more generally, solvable Lie algebras of vector fields can be{, in principle,} explicitly integrated, the above example suggests us the following procedure of integration of  $t$-dependent polyvector fields.

We need a certain family $V_0(\mathbb{R}^s)$ of polyvector fields ${\bf X}$ taking values in a finite-dimensional solvable Lie algebra of vector fields $V_0$. Using $t$-dependent generalised $t$-flows related to a $t$-dependent polyvector field taking values in a finite-dimensional solvable Lie algebra of vector fields $W$, we can produce integrable systems from those integrable systems related to $V_0(\mathbb{R}^s)$. This motivates the following scheme of integration

\begin{definition}
Let $W$ be a finite-dimensional Lie algebra of vector fields on $N$ and let $\cG^s(W)$ be the group of  generalised $t$-flows $g^\bX$  whose generators $\bX$ take values in $W$. Two $t$-dependent polyvector fields $\bY$ and ${\bf Z}$ on $N$, representing integrable first-order PDEs in normal form, are \emph{$W$-related} if there exists $g\in\cG^s(W)$ such that $g_\di \bY={\bf Z}$. If $\bY$ is {\it explicitly integrable} and $W$ is solvable, then we call ${\bf Z}$ \emph{quasi-solvable}.
\end{definition}

\begin{remark}
We do not define precisely what `explicit {integrability}' means as its sense depends on the context. Roughly speaking, the term indicates that we are able to write general solutions explicitly (e.g. through algebraic operations and integrals -- solvability by quadratures).
Since every vector field in a finite-dimensional transitive and solvable Lie algebra of vector fields is integrable by quadratures (cf. \cite{CFG16}), we conjecture that quasi-solvable first-order PDEs are actually explicitly {integrable}.
\end{remark}

\begin{example}\label{two} Let us consider the almost homogeneous ordinary differential equation
\begin{equation}\label{sys}
\frac{dy}{dt}=f\left(\frac{a_1t+b_1y+c_1}{a_2t+b_2y+c_2}\right),\quad y\in \mathbb{R},
\end{equation}
where $a_i,b_i,c_i$, with $i=1,2$, are real constants such that $a_1b_2-a_2b_1\neq 0$, and $f:\mathbb{R}\rightarrow \mathbb{R}$.
The diffeomorphism $\phi:(t,y)\in\mathbb{R}^2\mapsto (t,(y-y_0)/(t-t_0))\in \mathbb{R}^2$, where $(t_0,y_0)$ is the unique solution to the algebraic system
$$
\left\{\begin{array}{ccccc}
a_1t_0&+&b_1y_0&=&-c_1,\\
a_2t_0&+&b_2y_0&=&-c_2,\\
\end{array}\right.\quad a_1b_2-a_2b_1\neq 0,
$$
allows us to map (\ref{sys}) into the differential equation
\begin{equation}\label{int}
\frac{dy}{dt}=\frac{1}{t-t_0}\left[f\left(\frac{a_1+b_1y}{a_2+b_2y}\right)-y\right],
\end{equation}
whose solution can be immediately obtained by quadratures.
In other words, the $t$-dependent affine change of variables ${y}\mapsto (y-y_0)/(t-t_0)$, which can be understood as the generalised $t$-flow of a $t$-dependent vector field taking values in $W=\langle \partial_y, y\partial_ y\rangle$, maps the differential equation (\ref{sys}) onto (\ref{int}). Hence, the differential equation (\ref{sys}) can be considered as quasi-integrable relative to the finite-dimensional solvable Lie algebra $W$.
\end{example}


\begin{example}
As a final example, let us investigate an integrable system of first-order PDEs
$$
\frac{\partial y^i}{\partial t_\pi}=F_\pi^i\left(t,y\right),\quad i,\pi=1,\dots,s ,
$$
where $F^i_\pi(t'+ t,t' +y)=F^i_\pi(t,y)$ for each $i,\pi=1,\ldots,s$, every $t'\in \mathbb{R}^s$, and $t\in\mathbb{R}^s$. Then, a $t$-dependent change of variables $y=t+x$, related to the group of transformations $y=\sigma x+\lambda$ with $\sigma> 0$ and $t'\in \mathbb{R}$, transforms the above system of PDEs into  a new integrable one:
$$
\delta^i_k+\frac{\partial x^i}{\partial t_\pi}=F_\pi^i(t,t+x)=F_\pi^i(0,x)\Rightarrow \frac{\partial x^i}{\partial t_\pi}=F_\pi^i(0,x)-\delta^i_\pi,\qquad i,\pi=  1,\dots,s ,
$$
which is {a trivially integrable} autonomous system of PDEs. In fact, the integrability of the above system implies that the vector fields on $N$ of the form $(F^i_\pi(0,x)-\delta^i_\pi)\partial_{x^i}$, with $\pi=1,\ldots,s$, commute among themselves and there exist coordinates $u_1,\ldots, u_s$ rectifying {all them}  simultaneously. The knowledge of these coordinates leads to the immediate integration of the transformed system and, by inverting the $t$-dependent change of variables $y=t+x$, of the initial one.
\end{example}

The previous procedure is quite frequently applied in the literature on integrability conditions in a more or less clear way \cite{AS64,CL15,Ma16,HLM13}. Obviously, it gives rise to finding an infinite number of integrability conditions for differential equations. Nevertheless, we believe that the real problem consists in giving a procedure to determine when a certain differential equation can be obtained by means of the above approach. That is, the crucial question is to provide a geometric description of the family of $t$-dependent polyvector fields obtained by transforming a family of $t$-dependent polyvector fields taking values in a finite-dimensional Lie algebra of vector fields by means of the $t$-dependent changes of variables originated by the Lie group action of a solvable group.

A method to investigate the applicability of the above-mentioned procedure will be developed in forthcoming sections through the theory of quasi-Lie schemes \cite{CGL08,CL15}. More specifically, Section 4 studies the properties of PDE Lie systems and show several new applications of such systems. This is interesting as  we intend to map families of systems of PDEs to PDE Lie systems. Such families, which are proved to admit common $t$-dependent superposition rules, are analysed in Section 5 whereas Section 6 provides a characterisation of such families. Section 7 uses quasi-Lie schemes to map families of systems of PDEs into PDE Lie systems. Finally, Section 8 studies the so-called quasi-Lie invariants for PDEs, which help us to determine families of quasi-Lie systems of PDEs.

A list of some ODEs that can be approached by our methods is provided in Table \ref{table2}. Of course, one can use the proposed procedure not only to solve the equations or their families but also to find their geometrical properties, e.g. superposition rules.



\begin{table}[h!] {
  \noindent
\caption{{\small Types of ordinary differential equations approachable by quasi-Lie schemes and Lie algebras $W$. The number $n$ is an integer and $'$ stands for the derivative with respect to $t\in \mathbb{R}$.}}
\label{table2}
\medskip
\noindent\hfill
\begin{tabular}{ l   l l }
\hline
 & &\\[-1.5ex]
Form & Differential equations \\[+1.0ex]
\hline
     &\\[-1.5ex]
 $y'=a_0(t)+a_1(t)y+\ldots +a_n(t)y^n$, $n>2$&  Abel differential equations of the first kind \\[2pt]
 $y'=a_1(t)y+a_n(t)y^\alpha,\quad \alpha \in \mathbb{R}\backslash\{1,0\}$&  Bernoulli equations   \\[2pt]
 $y''=a(t)y'-a_n(t)y^\alpha,\quad \alpha\in \mathbb{R}$&  generalised Emden equations \\[2pt]
 $y'=f\left(\frac{a_1t+b_1y+c_1}{a_2t+b_2y+c_2}\right),\quad a_1b_2-a_2b_1\neq 0$& Almost homogeneous differential equations\\[2pt]
 $y'=f\left(t,y\right)$, $f(\lambda t,\lambda y)=\lambda^nf(t,y)$ & $n$-order homogeneous differential equations \\
 $y''=a(t)y'+a_2(t)y+a_3/y^3,\,\,a_3\in \mathbb{R}\backslash\{0\}$ & Dissipative Milne--Pinney equations\\
$y'=-y\frac{e^{i\alpha t}+y}{e^{i\alpha t}-y}$& Levner equation\\
 \\[2pt]
    \hline
\end{tabular}
\hfill}
\end{table}
\section{PDE Lie systems: basics and new examples}
To introduce quasi-Lie schemes and quasi-Lie systems, let us discuss now results concerning the so-called \emph{PDE Lie systems} \cite{CGM07}. Relevantly, most examples and applications of PDE Lie systems showed in this section are new and they can be applied in the study of B\"acklund transformations and other new problems. This is specially interesting as their study has been so far almost always purely theoretical (cf. \cite{CGLS14,Dissertationes,GR95,OG00}).

\begin{definition} A {\it PDE Lie system} ${\bf X}$ is a non-autonomous ($t$-dependent) system of first-order PDEs on $N$ in normal form admitting a {\it superposition rule}, i.e. there exists a function $\Phi:N^m\times N\rightarrow
	N$ such that the general solution to ${\bf X}$, let us say $x(t)$, can be written as
\begin{equation*}\label{SupPDE}
x(t)=\Phi\left({x_{(1)}}(t),\ldots,{x_{(m)}}(t);\lambda \right),
\end{equation*}
for a generic family $x_{(1)}(t),\ldots, x_{(m)}(t)$ of particular solutions and {$\lambda \in N$}.  We call $\Phi$ a \emph{superposition rule} for $\bX$.
\end{definition}


\begin{example}A {{\it Riccati  partial  differential equation}} is a non-autonomous system of PDEs of the form
\begin{equation}\label{PRE}
\left\{
\begin{array}{rl}
\dfrac{\partial u}{\partial t_1}&=b_1^{1)}(t)+b^{2)}_1(t) u+b^{3)}_1(t) u^2,\\
\dfrac{\partial u}{\partial t_2}&=b_2^{1)}(t)+b^{2)}_2(t) u+b^{3)}_2(t) u^2,\\
\end{array}\right.\quad t=(t_1,t_2)\in \mathbb{R}^2,\quad  u\in \mathbb{R},
\end{equation}
where $b_i^{1)},b_i^{2)},b_i^{3)}:\mathbb{R}^2\rightarrow \mathbb{R}$, with $i=1,2$, are arbitrary functions such that the ZCC condition is satisfied. {Riccati partial  differential} equations appear, for instance, in the study of Toda lattices and in particular cases of Wess-Zumino-Novikov-Witten (WZNW) (cf. \cite{FGRSZ99}).
A particular  {Riccati partial  differential}  equation satisfying the ZCC condition is given by\footnote{The system (\ref{BT}) solves a typo present in \cite[Eq. (10)]{PZ04}. It is simple to see that the compatibility condition of the expression of Eq. (10) of that book cannot give rise to the KdV equation.}
\begin{equation}\label{BT}\left\{
\begin{aligned}
\dfrac{\partial u}{\partial t_1}&=-\epsilon\frac{\partial^2 w}{\partial t_2^2}+2u\frac{\partial w}{\partial t_2}+2w\epsilon(u^2-w),\\
\dfrac{\partial u}{\partial t_2}&=\epsilon(-u^2+w),\qquad \epsilon=\pm 1,
\end{aligned}\right.
\end{equation}
where $w(t_1,t_2)$ stands for a particular solution to the Korteweg-de Vries (KdV) equation
$$
\frac{\partial w}{\partial t_1}-6w\frac{\partial w}{\partial t_2}+\frac{\partial^3w}{\partial t_2^3}=0.
$$
Every particular solution $u(t,x)$ of (\ref{BT}) becomes a solution of the modified Korteweg-de Vries equation $\partial_{t_1} u+\partial_{t_2}^3u-6u^2\partial_{t_2}u=0$ \cite{PZ04}. Moreover, the system of PDEs (\ref{BT}) can be understood as a certain type of Miura transformation for $\epsilon=1$, or, a B\"acklund transformation between the KdV equation and the modified KdV equation  (cf. \cite{PZ04}).

Since our  {Riccati partial  differential}  equation is a system of first-order PDEs in normal form assuming the ZCC condition, the value of a particular solution $ u(t)$ to (\ref{PRE}) at a single point $t_0$ determines $u(t)$
{on} an open neighbourhood of the point. Hence, if all particular solutions are assumed to be globally defined, then the space of solutions of a  {Riccati partial  differential} equation  can be parametrised by a unique real parameter $\lambda$.

It is known \cite{CGM07} that the general solution $ u(t)$ to (\ref{PRE}) can be written as
\begin{equation}\label{SupRicc}
 u(t)=\frac{ u_{(1)}(t)( u_{(3)}(t)- u_{(2)}(t))-\lambda\, u_{(2)}(t)( u_{(3)}(t)- u_{(1)}
(t))}{( u_{(3)}(t)- u_{(2)}(t))-\lambda\,( u_{(3)}(t)- u_{(1)}(t))}
\end{equation}
for different particular solutions $ u_{(1)}(t),  u_{(2)}(t),  u_{(3)}(t)$ of (\ref{PRE}) and $\lambda \in \mathbb{R}$. Therefore,
$
 u(t)=\Phi( u_{(1)}(t), u_{(2)}(t), u_{(3)}(t);k)
$
for
$\Phi:\mathbb{R}^3\times\mathbb{R}\rightarrow\mathbb{R}$ defined by
$$
\Phi( u_{(1)}, u_{(2)}, u_{(3)};\lambda ):=\frac{ u_{(1)}( u_{(3)}- u_{(2)})-\lambda\, u_{(2)}( u_{(3)}
- u_{(1)})}{( u_{(3)}- u_{(2)})-\lambda\,( u_{(3)}- u_{(1)})}
$$
is an example of a superposition rule for the Riccati partial differential equations.

\end{example}
\color{black}
Superposition rules for systems of first-order PDEs were studied in \cite{CGM07,OG00}. A theorem characterising them reads as follows.

\begin{theorem} {\bf (The PDE Lie--Scheffers Theorem \cite{CGM07,OG00})} An integrable $t$-dependent system of PDEs ${\bf X}$ possesses a
superposition rule if and only if
\begin{equation*}\label{LieDecom}
{\bf X}=\sum_{\pi=1}^s\sum_{l=1}^rb_{\pi,l} (t)Y_l\otimes e^\pi,
\end{equation*}
for a family $Y_1,\ldots,Y_r$ of vector fields on $N$ spanning an
$r$-dimensional real Lie algebra, a so-called {\it Vessiot--Guldberg
Lie algebra} for ${\bf X}$,
and $t$-dependent functions $\{b_{\pi,l}(t)\}_{\tiny \begin{aligned}l\!=\!1,\!\dots\!,r\\\\[-2em] \pi\!=\! 1,\!\dots\!,s\end{aligned} }$. 
\end{theorem}

\begin{example}
Let us analyse the system of PDEs
\begin{equation}\label{PRE3}
\left\{
\begin{aligned}
\displaystyle\frac{\partial u}{\partial t_1}&=\sin(u+g)-\frac{\partial f}{\partial t_1}\,,\\
\displaystyle\frac{\partial u}{\partial t_2}&=\sin(u+f)-\frac{\partial g}{\partial t_2}\,,
\end{aligned}\right.
\end{equation}
where $f,g$ are real functions depending on $t_1,t_2$. The previous system of PDEs satisfies the ZCC condition if and only if $A(t_1,t_2)=f(t_1,t_2)-g(t_1,t_2)$ satisfies the sine-Gordon equation \cite{Ra89}
$$\frac{\partial^2 A}{\partial t_1\partial t_2}=\sin(A)\,.$$
As $\sin(u+h)=\sin(u)\cos(h)+\cos(u)\sin(h)$, the system of PDEs (\ref{PRE3}) is related to a $t$-dependent polyvector field taking values in the Lie algebra, $V^{\rm sg}$, spanned by the vector fields
$$X_1=\partial_u\,,\quad X_2=\sin(u)\partial_u\,,\quad X_3=\cos(u)\partial_u,$$
which is isomorphic to $\mathfrak{sl}(2,\R)$. In virtue of the PDE Lie--Scheffers Theorem, (\ref{PRE3}) is a PDE Lie system and $V^{\rm sg}$ is one of its Vessiot--Guldberg Lie algebras.

\end{example}

\begin{example} Any particular solution $w(t_1,t_2)$ of the Liouville equation $\partial^2w/\partial t_1\partial t_2=ae^{\lambda w}$, with $a,\lambda\in \mathbb{R}$, is a particular solution of the system of PDEs of the form (see \cite[Section 3.5.1.2]{PZ04} for details):
\begin{equation}\label{Liu}
\left\{\begin{aligned}
\frac{\partial w}{\partial t_1}&=\frac{\partial u}{\partial t_1}-\frac{2}{\lambda}\exp\left(\frac\lambda 2(w+u)\right),\\
\frac{\partial w}{\partial t_2}&=-\frac{\partial u}{\partial t_2}-a\exp\left(\frac{\lambda}{2}(w-u)\right),\\
\end{aligned}\right.
\end{equation}
where $u(t_1,t_2)$ is a particular solution of $\partial^2u/\partial t_1\partial t_2=0$, and vice versa. It is immediate that the above system of PDEs is related to a $t$-dependent polyvector field taking values in the Vessiot--Guldberg Lie algebra $V_L=\langle X_1=\partial_w, X_2=e^{\lambda w/2}\partial_w\rangle$. Hence, (\ref{Liu}) is a PDE Lie system.
\end{example}

\begin{example}Let $G$ be a Lie group. Consider the reduction of the {\it Wess-Zumino-Novikov-Witten (WZNW) equations} given by \cite{FGRSZ99}
\begin{equation}\label{WZNW}
\frac{\partial \psi}{\partial t_{-\pi}}=-R_{\psi*}\lambda_{-\pi},\qquad\frac{\partial \psi}{\partial t_\pi}=L_{\psi*}\lambda_{\pi},\qquad \pi=1,\ldots,s,\quad \psi\in G,
\end{equation}
where $L_\psi:\psi_0\in G\mapsto \psi\psi_0\in G$, $R_\psi:\psi_0\in G\mapsto \psi_0\psi\in G$, and the functions $\lambda_{-\pi},\lambda_\pi:\mathbb{R}^{2s}\rightarrow \mathfrak{g}$ for $\pi=1,\ldots,s$ satisfy
$
\partial_{t_\pi}\lambda_{-\pi'}=\partial_{t_{-\pi'}}\lambda_{\pi}=0$, for $\pi,\pi'=1,\ldots,s.
$
The system of PDEs (\ref{WZNW}) also appears in the study of the Toda lattice and gives a generalisation of the Redheffer-Reid systems \cite{Re56,Re57,Re59}.

Let us write $t=(t_{-s},\ldots,t_{-1},t_1,\ldots,t_{s})$. Our reduced WZNZ equation can be related to a unique $t$-dependent polyvector field ${\bf X}$ on $G$ given by
$$
{\bf X}(t,\psi)=\sum_{\pi=1}^s[-R_{\psi*}\lambda_{-\pi}(t)\otimes e_{-\pi}+L_{\psi*}\lambda_{\pi}(t)\otimes e_{\pi}].
$$
The integrability of (\ref{WZNW}) 
 amounts to the well-known conditions \cite{FGRSZ99}
$$
\partial_{t_{-\pi}}\lambda_{-\pi'}-\partial_{t_{-\pi'}}\lambda_{-\pi}+[\lambda_{-\pi},\lambda_{-\pi'}]=0,\quad \partial_{t_{\pi}}\lambda_{\pi'}-\partial_{t_{\pi'}}\lambda_{\pi}+[\lambda_{\pi},\lambda_{\pi'}]=0,\quad \pi,\pi'=1,\ldots,s.
$$
The $t$-dependent polyvector field ${\bf X}$ takes values in the Lie algebra $V^{RL}$ spanned by  right- and left-invariant vector fields on $G$. Hence, (\ref{WZNW}) becomes a PDE Lie system.

If $G$ is commutative, then the multiplication of a particular solution $\psi_0(t)$ of (\ref{WZNW}) on the right by an element $\sigma\in G$ generates a new solution. Since the value of every particular solution to (\ref{WZNW}) is determined by its value at a fixed point, the general solution, $\psi(t)$, to the system of PDEs (\ref{WZNW}) can be written as
$$
\psi(t)=R_\sigma\psi_0(t),\qquad \sigma\in G.
$$
This gives rise to a superposition rule $\Phi:(\psi_0,\sigma)\in G\times G\mapsto R_\sigma\psi_0\in G$ for ${\bf X}$ depending on just one particular solution.
\end{example}

\section{Families of PDEs, $t$-dependent superpositions rules and foliations.}\label{TDS}

This section presents the concept of common $t$-dependent superposition rules for a family of $t$-dependent systems of first-order PDEs. We study the geometry of this structure and  relate this concept to horizontal foliations in certain fibrations. 

Consider a family of $t$-dependent systems of
integrable first-order PDEs on $N$ parameterised by elements $\alpha$ of a set $\Lambda$ given by
\begin{equation}\label{eq1}
\frac{\partial x^i}{\partial t_\pi}=Y^i_{\pi,\alpha}(t,x),\qquad i=  1,\dots,n ,\qquad \pi=1,\dots,s ,\qquad \alpha\in\Lambda.
\end{equation}
In applications, $\Lambda$ is often a finite set or a certain family of $t$-dependent functions (cf. \cite{CGL09}). 
We will hereafter view $\mathbb{R}^s\times N^{m+1}$ and $\mathbb{R}^s\times N$ as fibre bundles relative to their standard projections onto $\mathbb{R}^s$.

\begin{definition}\label{CSRDef}{\rm A {\it common $t$-dependent superposition rule} depending on $m$ particular solutions for a family of $t$-dependent systems of PDEs (\ref{eq1}) is  a mapping $\Phi:\mathbb{R}^s\times
N^{m}\rightarrow N$ of the form
\begin{equation}\label{TSupRul}
x=\Phi(t,x_{(1)},\ldots,x_{(m)};\lambda),
\end{equation}
such that the general solution, $x(t)$, of an arbitrary system ${\bf Y}_\alpha$ of (\ref{eq1}) can be written as
$$
x(t)=\Phi(t,x_{(1)}(t),\ldots,x_{(m)}(t);\lambda),
$$
where $x_{(1)}(t),\ldots, x_{(m)}(t)$ is a generic set of particular solutions of ${\bf Y}_\alpha$ and
$\lambda\in N$. A family of systems (\ref{eq1})
admitting a common $t$-dependent superposition rule is called a {\it PDE Lie family}.}
\end{definition}


\begin{example}
All time-dependent systems of first-order ordinary differential equations sharing a common Vessiot--Guldberg Lie algebra admit a common time-independent superposition rule \cite[Proposition 2.9]{CGL12}.

\end{example}

Our aim now is to characterise common $t$-dependent superposition rules for a family of systems of PDEs as a certain type of flat connection. The proof of this fact is given in Proposition \ref{Connection}, which is a generalisation of \cite[Proposition 6]{CGL09}. Our proof follows the same ideas of that work while taking into account several additional complications due to the need of dealing with $t$-dependent polyvector fields. 

\begin{definition}
{\rm Given a $t$-dependent polyvector field  ${\bf Y}=\sum_{\pi'=1}^s\sum_{i=1}^nY_{\pi'}^i(t,x)\partial/\partial x^i\otimes e^{\pi'}$ on $N$,  its {\it prolongation} to $\mathbb{R}^s\times N^{m+1}$ is the polyvector field on
$\mathbb{R}^s\times N^{m+1}$ given by
\begin{equation}\label{prolongation}
\widehat {\bf Y}^{(m)}(t,x_{(0)},\ldots,x_{(m)})=\sum_{\pi'=1}^s\sum_{a=0}^{m}\sum_{i=1}^nY_{\pi'}^ i(t,x_{(a)})\pd{}{x^ i_{(a)}}\otimes e^{\pi'},
\end{equation}
and its
{\it $t_\pi$-prolongation} to $\mathbb{R}^s\times N^{m+1}$, where $\pi=1,\ldots,s$, is the polyvector field on
$\mathbb{R}^s\times N^{m+1}$ of the form
\begin{equation}\label{prolongation2}
\widetilde {\bf Y}^{[\pi]}(t,x_{(0)},\ldots,x_{(m)})=\sum_{\pi'=1}^s\left(\pd{}{t_\pi}+\sum_{a=0}^{m}\sum_{i=1}^nY_{\pi'}^ i(t,x_{(a)})\pd{}{x^ i_{(a)}}\right)\otimes e^{\pi'}.
\end{equation}
The $t_\pi$-prolongation of ${\bf Y}$ to $\mathbb{R}^s\times N$, namely $\bar {\bf Y}^{[\pi]}=\sum_{\pi'=1}^s\left(\partial/\partial t_\pi+\sum_{i=1}^nY^i_{\pi'}(t,x)\partial/\partial x^i\right)\otimes e^{\pi'}$, is called its {\it $t_\pi$-autonomisation}.
}
\end{definition}

The superscript $(m)$ in (\ref{prolongation}) will be dropped when its value will be clear from context. The above definitions retrieve for $s=1$, the definitions for prolongations and time-prolongations of time-dependent vector fields given in \cite{CGL09}. By dropping the sum over $\pi'$ and the index $\pi'$ in (\ref{prolongation}) and (\ref{prolongation2}), we obtain the prolongations and $t$-prolongations of a $t$-dependent vector field $Y=\sum_{i=1}^nY^i(t,x)\partial/\partial x^i$.
The following two lemmata are straightforward.
\begin{lemma}\label{PureProl} Given two $t$-dependent vector fields $X$ and $Y$ on $N$ and their $t_\pi$- and $t_{\pi'}$-prolongations $\widetilde{X}^{[\pi]}$, $\widetilde{Y}^{[\pi']}$ to $\mathbb{R}^s\times N^{m+1}$,
the Lie bracket $\left[\widetilde X^{[\pi]},\widetilde Y^{[\pi']}\right]$ is the prolongation to $\mathbb{R}^s\times N^{m+1}$
of a $t$-dependent vector field $Z$ on $N$, i.e. $[\widetilde X^{[\pi]},\widetilde Y^{[\pi']}]=\widehat Z$  for some $t$-dependent vector field $Z$ on $N$.

\end{lemma}
\begin{lemma}\label{Aut} Given a family of $t$-dependent vector fields, $X_1,\ldots,X_r$,  on $N$,
their $t_\pi$-autonomisations $\bar X_j^{[k]}$, $\pi=1,\dots,s$ and $j=1,\ldots,r$, satisfy the relations
$$
[\bar X^{[\pi]}_o, \bar X^{[\pi']}_p](t,x)=\sum_{\pi''=1}^s\sum_{u=1}^rf^{\pi \pi' u}_{op\pi''}(t)\bar X^{[\pi'']}_u(t,x)\,,\qquad o,p=1,\ldots,r,\quad \pi,\pi'=1,\ldots,s,
$$
for some $t$-dependent functions $f^{\pi \pi'u}_{op\pi''}:\mathbb{R}^s\rightarrow\mathbb{R}$, if and only if their
$t_\pi$-prolongations to $\mathbb{R}^s\times N^{m+1}$ satisfy
analogous relations, i.e.
$$
\left[\widetilde X^{[\pi]}_o, \widetilde X^{[\pi']}_p\right](t,x)=\sum_{\pi''=1}^s\sum_{u=1}^rf^{\pi \pi'u}_{op\pi''}(t)\widetilde X^{[\pi'']}_u(t,x)\,\qquad o,p=1,\ldots,r,\quad \pi,\pi'=1,\ldots,s.
$$
Moreover,
$\sum_{u=1}^rf^{\pi \pi'u}_{op\pi''}(t)=0$ for all $\pi,\pi',\pi''=  1,\dots,s $ and $o,p=  1,\dots,r$.
\end{lemma}

We now turn to proving the main result of this section.

\begin{proposition}\label{Connection}
A common $t$-dependent superposition rule $\Phi:\mathbb{R}^s\times N^m\rightarrow N$ for a PDE Lie family (\ref{eq1}) amounts to a horizontal foliation relative to the projection
$$
{\rm pr}:(t,x_{(0)},\ldots,x_{(m)})\in\mathbb{R}^s\times N^{m+1}\mapsto
(t,x_{(1)},\ldots,x_{(m)})\in \mathbb{R}^s\times N^m$$
such that the vector
fields $\{\widetilde Y^{[\pi]}_{\pi,\alpha}\}_{
\alpha\in\Lambda,\pi= 1,\dots,s }$  are tangent to their leaves.
\end{proposition}

\begin{proof} Let us prove the direct part. The mapping
$\Phi(t,x_{(1)},\ldots,x_{(m)};\cdot):
N\mapsto  N$,
$x_{(0)}=\Phi(t,x_{(1)},\ldots,x_{(m)};\lambda)$, is regular at a generic point $(t,x_{(1)},\ldots, x_{(m)}) \in
\mathbb{R}^s\times N^m$. The Implicit Function Theorem ensures that this map can be inverted, which gives rise to a  function $\Psi: \mathbb{R}^s\times N^{m+1}\rightarrow N$ such that
$
\lambda=\Psi(t,x_{(0)},\ldots, x_{(m)}),
$
where $\lambda$ is the unique point in $N$ satisfying that
$
x_{(0)}=\Phi(t,x_{(1)},\ldots,x_{(m)};\lambda).
$
Consequently, $\Psi$ determines  an $n$-codimensional foliation $\mathfrak{F}$ on $\mathbb{R}^s\times N^{m+1}$ whose leaves are the level sets of $\Psi$.

Let $x_{(1)}(t),\ldots,x_{(m)}(t)$ be particular solutions of a particular ${\bf Y}_{\hat\alpha}$. Since $\Phi$ is a $t$-dependent common superposition rule for the  $\{{\bf Y}_\alpha\}_{\alpha\in \Lambda}$, there exists for every particular solution $x_{(0)}(t)$ of ${\bf Y}_{\hat\alpha}$ a unique $\lambda \in N$ such that $x_{(0)}(t)=\Phi(t,x_{(1)}(t),\ldots,x_{(m)}(t),\lambda)$,  Hence,
$\Psi(t,x_{(0)}(t),\ldots, x_{(m)}(t))$ is constant for any $(m+1)$-tuple of particular solutions of ${\bf Y}_{\hat\alpha}$. Then, the foliation determined by $\Psi$ is invariant
under the permutation of its $(m+1)$ arguments $x_{(0)},\ldots,x_{(m)}$. Anyhow, it is worth noting, like in the case of the superposition rule (\ref{SupRicc}), that in general
$$
\Psi(t,x_{(0)},\ldots,x_{(i)},\ldots,x_{(j)},\ldots,x_{(m)})\neq \Psi(t,x_{(0)},\ldots,x_{(j)},\ldots,x_{(i)},\ldots,x_{(m)}).
$$
Differentiating
$\Psi(t,x_{(0)}(t),\ldots, x_{(m)}(t))$ relative to $t_\pi$ for $\pi= 1,\dots,s $, we obtain
\begin{equation}\label{Psifi}
\pd{\Psi^j}{t_\pi}+\sum_{a=0}^{m}\sum_{i=1}^nY_{\pi,\alpha}^i(t,x_{(a)}(t))\pd{\Psi^j}{x^i_{(a)}}=0,\qquad
j=1,\ldots,n,\quad \pi=1,\dots,s ,\quad \alpha\in\Lambda,
\end{equation}
where $\Psi(t,\cdot)=(\Psi^{1}(\cdot),\ldots,\Psi^{n}(\cdot))$. In consequence, $\Psi^1,\ldots,\Psi^n$ are common first-integrals for all the
vector fields $\{\widetilde Y^{[\pi]}_{\pi,\alpha}\}_{\pi=1,\ldots,s,\alpha\in\Lambda}$. 
Therefore, all the vector fields $\widetilde Y^{[\pi]}_{\pi,\alpha}$ are tangent to the
leaves of $\mathfrak{F}$.

Let us see that the foliation $\mathfrak{F}$ is horizontal relative to ${\rm pr}$. If $\mathfrak{F}_\lambda$ is the level set of
$\Psi$ corresponding to a certain $\lambda=(\lambda_1,\ldots,\lambda_n)\in N$ and
$(t,x_{(1)},\ldots,x_{(m)})\in\mathbb{R}^s\times N^m$, then there is only one point
$x_{(0)}\in N$ such that  $(t,x_{(0)},x_{(1)},\ldots,x_{(m)})\in\mathfrak{F}_\lambda$. In consequence, ${\rm pr}$
induces a local diffeomorphism between
$\mathfrak{F}_\lambda$ and $\mathbb{R}^s\times N^m$. 

Let us prove the converse part: a foliation $\mathfrak{F}$ in $\mathbb{R}^s\times N^{m+1}$ that is horizontal relative to ${\rm pr}:\mathbb{R}^s\times N^{m+1}\rightarrow N^m$ defines a common $t$-dependent superposition rule. Indeed, given a point $x_{(0)}\in N$ and $m$ particular solutions,
$x_{(1)}(t),\ldots,x_{(m)}(t)$, for any system ${\bf Y_{\hat \alpha}}$ of the family (\ref{eq1}),  one has that  $x_{(0)}(t)$ is the unique curve in
$N$ such that the points of the curve
\begin{equation}\label{curve}
(t,x_{(0)}(t),x_{(1)}(t),\ldots, x_{(m)}(t))\subset
\mathbb{R}^s\times N^m
\end{equation}
belong to the same leaf as the point $(0,x_{(0)}(0),x_{(1)}(0),\ldots,x_{(m)}(0))$.
Therefore
$$
\Phi(t,x_{(1)}(t),\ldots,x_{(m)}(t);\lambda)=x_{(0)}(t).
$$
Recall that $(t,x_{(1)}(t),\ldots,x_{(m)}(t))$ is an integral submanifold of the prolongation of ${\bf Y}_{\hat\alpha}$
to $\mathbb{R}^s\times N^m$, i.e. $\widehat {\bf Y}^{(m)}_{\hat\alpha}$, by construction. As the foliation in $\mathbb{R}^s\times N^{m+1}$ is horizontal relative to ${\rm pr}$, the lift of the $\widehat {\bf Y}^{(m)}_{\hat\alpha}$ to $\mathbb{R}^s\times N^{m+1}$ is $\widehat {\bf Y}^{(m+1)}_{\hat \alpha}$, namely the prolongation of ${\bf Y}_{\hat \alpha}$ to $\mathbb{R}^s\times N^{m+1}$. Thus, $(t,x_{(0)}(t),\ldots,x_{(m)}(t))$ is an integral submanifold of $\widehat {\bf Y}^{(m+1)}_{\hat \alpha}$ and, in consequence, $(t,x_{(0)}(t))$ is a particular solution to ${\bf Y}_{\hat{\alpha}}$. This causes the mapping $\Phi:\mathbb{R}^s\times N^{m+1}\rightarrow N$ to be a superposition rule for the family of systems $\{{\bf Y}_\alpha\}_{\alpha\in \Lambda}$.

\end{proof}

\section{Generalised PDE Lie--Scheffers Theorem.}\label{GLT}

Although Proposition \ref{Connection} offers an elegant geometric characterisation of families of systems of PDEs admitting a common $t$-dependent superposition rule, it is in general inadequate to determine the existence of such a superposition for a concrete family (\ref{eq1}). This leads us here to
characterise PDE Lie families via an easily verifiable
condition based on the properties of the $t$-dependent polyvector {fields. }
Such a criterion is formulated as the generalised PDE Lie--Scheffers Theorem. To prove it, let us show first the following lemma, which represents a generalisation to the realm of systems of PDEs of Lemma 1 in \cite{CGM07} and Lemma 9 in \cite{CGL09}. In the following $\delta_{\pi_j}^{\pi'}$ stands for the {\it Kronecker's delta}.

{\begin{lemma}\label{Prol} Let $Y_1,\ldots,Y_r$ be $t$-dependent vector fields on $N$ whose $t_\pi$-prolongations $\widetilde Y^{[\pi_1]}_1,\ldots,\widetilde Y^{[\pi_r]}_r$ to $\mathbb{R}^s\times N^{m+1}$, $\pi_1,\ldots,\pi_r=  1,\dots,s $, are such that the
${\rm pr}_*(\widetilde Y^{[\pi_j]}_j)$ are linearly independent at a generic point of $\mathbb{R}^s\times N^m$.
Then, if $b_1,\ldots,b_r\in C^{\infty}(\mathbb{R}^s\times N^{m+1})$, then
$$
\sum_{j=1}^r b_j\widetilde Y^{[\pi_j]}_j=\widehat Y^{(m+1)}\,\quad \left(\text{resp.}\ \sum_{j=1}^r b_j\widetilde Y^{[\pi_j]}_j=\widetilde Y^{[\pi_Y]},\,\,\pi_Y=1,\ldots,s\right)
$$
for a $t$-dependent vector field $Y$ on $N$ if and only if the functions $b_j$ depend on $t$ only and $\sum_{j=1}^r\delta_{\pi_j}^{\pi'}b_j=0$  (resp. $\sum_{j=1}^r\delta^{\pi'}_{\pi_j}b_j=\delta_{\pi_Y}^{\pi'}$) for $\pi'=1,\dots,s $.
\end{lemma}
}
\begin{proof}
We will only prove the case $\sum_{j=1}^r b_j\widetilde Y^{[\pi_j]}_j=\widehat Y^{(m+1)}$ since the proof for the other one is identical. The expressions for the $\widetilde{Y}_j^{[\pi_j]}$ and $Y$ read in coordinates,
\begin{equation*}
    \widetilde Y^{[\pi_j]}_j=
    \frac{\partial}{\partial t_{\pi_j}}+\sum_{i=1}^n\sum_{a=0}^mA^i_{j}(t,x_{(a)})\pd{}{x^i_{(a)}},\qquad j=1,\dots,r ,\qquad Y=\sum_{i=1}^nB^i(t,x)\frac{\partial}{\partial x^i}
 \end{equation*}
 for certain functions $A_j^i,B^i: \mathbb{R}^s\times N\rightarrow \mathbb{R}$.
Then, $\sum_{j=1}^r b_j\widetilde Y^{[\pi_j]}_j=\widehat Y^{(m+1)}$ amounts to
{\footnotesize
\begin{multline*}
\sum_{j=1}^r\left(\sum_{i=1}^nb_j(t,x_{(0)},
\ldots,x_{(m)})A^i_{j}(t,x_{(a)})\pd{}{x^i_{(a)}}
+\sum_{{\pi'}=1}^s \delta_{\pi'}^{\pi_j}b_j(t,x_{(0)},\ldots,x_{(m)})\pd{}{t_{{\pi'}}}\right)=\sum_{i=1}^nB^i(t,x_{(a)})\frac{\partial}{\partial x_{(a)}^i},
\end{multline*}}for every $a=0,\ldots,m$. The above holds if and only if  there exist functions $B^i:\mathbb{R}^s\times N\rightarrow\mathbb{R}$,
with $i= 1,\dots,n $, such that
\begin{equation*}
\left\{\begin{aligned}
&\sum_{j=1}^r b_j(t,x_{(0)},\ldots,x_{(m)})A^i_j(t,x_{(a)})=B^i(t,x_{(a)}),\\
&\sum_{j=1}^r \delta_{\pi_j}^{\pi'}b_j(t,x_{(0)},\ldots,x_{(m)})=0,
\end{aligned}\right.\qquad \begin{gathered}a=0,\dots,m , \quad i=  1,\dots,n ,\\ {\pi'}= 1,\dots,s .\end{gathered}
\end{equation*}If the functions $b_1,\ldots,b_r$ are $t$-dependent only and
$\sum_{j=1}^r
\delta_{\pi'}^{\pi_j}b_j=0$ for every ${\pi'}= 1,\dots,s$, the above conditions are obeyed and $\sum_{j=1}^r b_j \widetilde Y^{[\pi_j]}_j$ is the diagonal prolongation to
$\mathbb{R}^s\times N^{m+1}$ of the $t$-dependent vector field
$Y=\sum_{i=1}^nB^i(t,x)\partial/\partial x^i.$

Conversely, if $\sum_{j=1}^r b_j \widetilde Y^{[\pi_j]}_j$ is a diagonal prolongation for a $t$-dependent vector field
on $N$, then the functions $b_j(t,x_{(0)},\ldots,x_{(m)})$, with $j=1,\ldots,r$, solve the system of
$s+ m\cdot n$ linear equations in the $r$ unknown variables $u_j$ given by:
\begin{equation}\label{GS}
\sum_{j=1}^r u_j A^i_{j}(t,x_{(a)})=B^i(t,x_{(a)}),\qquad
\sum_{j=1}^r \delta_{\pi_j}^{\pi'}u_j=0,
\end{equation}
where $a=1,\dots,m$, $i=1,\dots,n $, and ${\pi'}= 1,\dots,s $. Since the ${\rm pr}_*(\widetilde
Y^{[\pi_j]}_j)$, with $j=1,\dots,r$, are linearly independent by assumption, the solutions $u_\alpha$ are uniquely determined and they depend on the $B^i(t,x_{(a)})$ for $a=1,\ldots,m$ and $i=1,\ldots,n$.  Hence, the $u_j$ depend on the variables $\{t, x_{(1)},\ldots,x_{(m)}\}$ and they do not depend on $x_{(0)}$.
Since diagonal prolongations are invariant with respect to the symmetry group $S_{m+1}$ acting on
$N^{m+1}$ in the obvious way, the functions
$u_j=b_j(t,\ldots,x_{(m)})$, with $j=1,\dots,r$, must satisfy such a symmetry. Since they do not
depend on $x_{(0)}$, they cannot depend on the variables $\{x_{(1)},\ldots,x_{(m)}\}$ neither and they become functions
depending on $t$ only.
\end{proof}

\begin{theorem}{\bf (Generalised PDE Lie--Scheffers Theorem)} \label{MT}
The family of systems (\ref{eq1}) on $N$ admits a {\it common $t$-dependent superposition rule} if and only if the
vector fields of the family $\{\bar Y^{[\pi]}_{\pi,\alpha}\}_{\pi= 1,\dots,s ,\alpha\in\Lambda}$ can be written as
\begin{equation*}\label{MainDesc}
\bar Y^{[\pi]}_{\pi,\alpha}(t,x)=\sum_{j=1}^rb^{[\pi]}_{\alpha j}(t)\bar X^{[\pi_j]}_j(t,x),\qquad  \pi= 1,\dots,s ,\qquad \alpha\in\Lambda,
\end{equation*}
where the $b^{[\pi]}_{\alpha j}$ are $t$-dependent functions, the $\pi_j$ are certain integers between $1$ and $s$,   
and the $X_1,\ldots, X_r$ are
$t$-dependent vector fields on $N$ satisfying
\begin{equation*}\label{condition}
[\bar X^{[\pi_j]}_j,\bar X^{[\pi_k]}_k](t,x)=\sum_{l=1}^rf_{jk}^l (t)\bar X^{[\pi_l]}_l(t,x),\qquad j, k= 1,\dots,r,
\end{equation*}
for some functions $f_{jk}^l :\mathbb{R}^s\rightarrow\mathbb{R}$, with {$j,k,l=1,\dots,r$.} We call the family of
autonomisations, $\bar X^{[\pi_1]}_1,\ldots,\bar X^{[\pi_r]}_r$, a {\sl system of generators of the PDE Lie family}.
\end{theorem}

\begin{proof}

Let us prove the direct part. Assume that the family of systems (\ref{eq1}) admits a {common $t$-dependent superposition rule}. Let $\mathfrak{F}$ be its associated $n$-codimensional horizontal foliation ensured by Proposition \ref{Connection}. Again in view of Proposition \ref{Connection}, the vector fields
$\{\widetilde Y^{[\pi]}_{\pi,\alpha}\}_{\pi= 1,\dots,s ,\alpha\in\Lambda}$ are tangent to the leaves of the foliation $\mathfrak{F}$ and
span a distribution $\mathcal{D}_0$ on $\mathbb{R}^s\times N^{m+1}$. Let $A$ be the Lie algebra spanned by the $\{\widetilde Y^{[\pi]}_{\pi,\alpha}\}_{\pi= 1,\dots,s ,\alpha\in\Lambda}$ and all their possible Lie brackets, i.e.
\begin{equation}\label{family}
\widetilde Y^{[\pi]}_{\pi,\alpha}, \,\, [\widetilde Y^{[\pi']}_{\pi',\alpha},\widetilde Y^{[\pi'']}_{\pi'',\beta}],\,\,
[\widetilde Y^{[\pi]}_{\pi,\alpha},[\widetilde Y^{[\pi']}_{\pi',\beta},\widetilde Y^{[\pi'']}_{\pi'',\gamma}]],\ldots\,
\quad\,\pi,\pi',\pi'',\ldots= 1,\dots,s ,\,\, \alpha,\beta,\gamma,\ldots\in \Lambda.
\end{equation}
All the elements of $A$ are tangent to the leaves of $\mathfrak{F}$. Hence, there exist
up to $m\cdot n+s$ linearly independent ones at a generic point of $\mathbb{R}^s\times N^{m+1}$ and
 they must span an involutive generalised distribution $\mathcal{D}$ whose leaves are of dimension $r\leq
m\cdot n+s$.

Take now a set $\mathcal{A}$ of elements of $A$ that are linearly independent at a generic point and span the distribution $\mathcal{D}$ in a neighbourhood of a regular point of this foliation. By construction, at least one of
them must be of the form $\widetilde X_\pi^{[\pi]}$ for a certain $t$-dependent vector field $X_\pi$ on $N$ and  every possible $\pi= 1,\dots,s $.
In view of the form of the family (\ref{family}), those elements of $\mathcal{A}$ not being $t_\pi$-prolongations
are just prolongations. Therefore, if we add some $\widetilde X^{[\pi]}_\pi$ to those elements of $\mathcal{A}$ not being just
prolongations, we can redefine $\mathcal{A}$ in such a way that its elements are
$t_\pi$-prolongations $\widetilde X^{[\pi_1]}_1,\ldots,\widetilde X^{[\pi_r]}_r$ for certain $\pi_1,\ldots, \pi_r\in  1,\dots,s $ and $r\leq m\cdot n+s$. In consequence, $\mathcal{D}$ is
locally spanned near regular points by $t_\pi$-prolongations. As
 $\mathcal{D}$ is involutive, there exist $r^3$ real functions $f_{jk}^l $, with
$j,k,l=1,\dots,r,$ on $\mathbb{R}^s\times N^{m+1}$ such that
$$
[\widetilde X^{[\pi_j]}_j,\widetilde X^{[\pi_k]}_k]=\sum_{l=1}^rf_{jk}^l \widetilde X^{[\pi_l]}_l, \qquad j,k= 1,\dots,r.
$$
The left side of the above equalities are prolongations and their projections onto $\mathbb{R}^s\times N^m$ are linearly independent at each point because the leaves of $\mathfrak{F}$ are locally diffeomorphic among themselves and their projections to $\mathbb{R}^s\times N^m$. Then, Lemma \ref{Prol} ensures that $r^3$ functions $f_{jk}^l $ depend on time only and $\sum_{l=1}^n\delta^\pi_{\pi_l}f_{jk}^l =0$ for every $\pi=1,\dots,s $ and $j,k=1,\ldots,r$. Next, Lemma
\ref{Aut} implies that
$$
[\bar X^{[\pi_j]}_j,\bar X^{[\pi_k]}_k](t,x)=\sum_{l=1}^rf_{jk}^l (t)\bar X^{[\pi_l]}_l(t,x), \qquad j,k=1,\ldots,r.
$$
Since the vector fields $\{\widetilde Y^{[\pi]}_{\pi,\alpha}\}_{\alpha\in\Lambda,\pi=1,\ldots,s}$ take values  in the distribution $\mathcal{D}$, there exist some functions $b^{[\pi]}_{\alpha j}\in C^{\infty}(\mathbb{R}^s\times N^{m+1})$ such that $\widetilde Y^{[\pi]}_{\pi,\alpha}=\sum_{j=1}^rb^{[\pi]}_{\alpha j}\widetilde X^{[\pi_j]}_j$ for every $\alpha\in\Lambda$ and $\pi= 1,\dots,s $. In consequence, Lemma \ref{Prol} ensures that the functions $b^{[\pi]}_{\alpha j}$ depend only on $t$, namely $b^{[\pi]}_{\alpha j}=b^{[\pi]}_{\alpha j}(t)$. Therefore,
$$
\widetilde Y^{[\pi]}_{\pi,\alpha}=\sum_{j=1}^rb^{[\pi]}_{\alpha j}\widetilde X^{[\pi_j]}_j\Longrightarrow\bar Y^{[\pi]}_\alpha(t,x)=\sum_{j=1}^rb^{[\pi]}_{\alpha j}(t)\bar X^{[\pi_j]}_j(t,x),\quad \pi= 1,\dots,s,\quad \alpha\in\Lambda.
$$

Let us now turn to proving the converse of our theorem. Assume that
$$
\bar Y^{[\pi]}_{\pi,\alpha}(t,x)=\sum_{j=1}^rb^{[\pi]}_{\alpha j}(t) \bar X^{[\pi_j]}_j(t,x)
$$
for certain functions $b_{\alpha j}^{[\pi]}$ and $t_\pi$-autonomisations $\bar X^{[\pi_1]}_1,\ldots,\bar X^{[\pi_r]}_r$ such that
$$
[\bar X^{[\pi_j]}_j,\bar X^{[\pi_k]}_k](t,x)=\sum_{l=1}^rf_{jk}^l (t)\bar X^{[\pi_l]}_l(t,x), \qquad j,k= 1,\dots,r
$$
for certain $t$-dependent functions $f^j_{kl}$, with $j,k,l=1,\ldots,r$.
Lemma
\ref{Aut} ensures that the $t_\pi$-prolongations $\widetilde X^{[\pi_1]}_1,\ldots, \widetilde X^{[\pi_r]}_r$ to every $\mathbb{R}^s\times N^m$ span an involutive distribution
$\mathcal{D}$ for any $m$. Furthermore, the rank of $\mathcal{D}$
must be not greater than $r$ and therefore, for $m$ big enough, this distribution is at least $n$-codimensional and
it gives rise to a foliation $\mathfrak{F}_0$ which is horizontal with respect to the projection ${\rm pr}$.
If the codimension of $\mathfrak{F}_0$ is bigger than $n$, then $\mathfrak{F}_0$ can be enlarged so as to obtain an
$n$-codimensional foliation $\mathfrak{F}$, still horizontal with respect to the map ${\rm pr}$. The foliation $\mathfrak{F}$ leads in view of Proposition \ref{Connection} to a
common $t$-dependent superposition rule for the family (\ref{eq1}).
\end{proof}

\begin{note} When $\{{\bf Y}_\alpha\}_{\alpha\in \Lambda}$ is a standard family of $t$-dependent vector fields, Theorem \ref{MT} retrieves the theorem characterising Lie families given in \cite[Theorem 10]{CGL09}.
\end{note}

The following corollary applies the above theorem to PDE Lie systems to extend the Lie condition \cite{Dissertationes} to systems of PDEs.

\begin{corollary}\label{CorFin} Every family of $t$-dependent systems of PDEs on $N$ taking values in a Vessiot---Guldberg Lie algebra $V$ admits an autonomous, i.e. $t$-independent, common superposition rule depending on $m$ particular solutions with $\dim V\leq mn$.
\end{corollary}
\begin{proof} Let us apply the converse part of the proof of Theorem \ref{MT}. If $\{Y_1,\ldots, Y_r\}$ is a basis of $V$, then the distribution $\mathcal{D}$ on $\mathbb{R}^s\times N^m$ is spanned by the $t_\pi$-prolongations
\begin{equation}\label{DisD}
\frac{\partial}{\partial_{t_1}}+\sum_{a=1}^mY_i(x_{(a)}),\frac{\partial}{\partial_{t_1}},\ldots, \frac{\partial}{\partial_{t_s}},\qquad i=1,\ldots,r.
\end{equation}
These vector fields span a Lie algebra of dimension $s+r$. To obtain a common $t$-dependent superposition rule, $m$ must be such that the vector fields (\ref{DisD}) must be linearly independent at a generic point on $\mathbb{R}^s\times N^m$. Hence, $s+r\leq s+m\cdot n$ and $r\leq m\cdot n$ as in the case of standard Lie systems. The $n$-codimensional foliation on $\mathbb{R}^s\times N^{m+1}$ obtained by the $t_\pi$-prolongations to this manifold of (\ref{DisD}) is such that the vector fields $\partial_{t_i}$, with $i=1,\ldots,r$ are tangent to the leaves. In this case, the mapping $\Psi:\mathbb{R}^s\times N^{m+1}\rightarrow \mathbb{R}^n$ is independent of the
variables $t_1,\ldots,t_s$ and the induced superposition rule $\Phi$ obtained from $\Psi$ becomes autonomous.
\end{proof}


\section{Quasi-Lie systems and schemes for PDEs}\label{Theory}
The theory of quasi-Lie schemes \cite{CGL08,Dissertationes} provides results on the transformation properties of $t$-dependent
vector fields by a certain kind of $t$-dependent changes of variables associated with generalised $t$-flows of $t$-dependent vector fields. This can be employed to analyse the integrability properties and time-dependent superposition rules of systems of ODEs \cite{CGL08,CLL08Emd,CL08Diss}.
This section generalises the theory of quasi-Lie schemes to PDES and relates it to the theory of PDE Lie families \cite{CGL09}. Let us start by defining quasi-Lie systems.

\begin{definition}{\rm A {\it quasi-Lie system} is a pair $({\bf X},g)$, where ${\bf X}$ is a $t$-dependent polyvector field on $N$ (the {\it system}) and $g$ is a generalised $t$-flow  on $N$ (the {\it control}), satisfying that
$g_\di \bX$ is a PDE Lie system.}
\end{definition}

Observe that since $g_\di \bX$ is a PDE Lie system, is therefore integrable by assumption and $\bX$ must be integrable also.
The control $g$ of a quasi-Lie system $({\bf X},g)$ allows for the construction of one of the $t$-dependent superposition rules associated with ${\bf X}$ as shown in the following theorem.

\begin{theorem}\label{t4} Every quasi-Lie system $(\bX,g)$ on $N$ admits a
$t$-dependent superposition rule $\Upsilon:\mathbb{R}^s\times N^{m+1}\rightarrow N$ given by
\begin{equation}\label{eq9}
x=\Upsilon(t,x_{(1)},\ldots, x_{(m)};\lambda)=g_t^{-1}\circ \Phi(g_t(x_{(1)}),\ldots,g_t(x_{(m)});\lambda),
\end{equation}
where $\Phi:N^{m}\times N\rightarrow N$ is a superposition rule for the PDE Lie
system $g_\di\bX$.
\end{theorem}

\begin{proof}
Since $g_\di{\bf X}$ is a PDE Lie system with a superposition rule $\Phi$, its general solution, $\bar x(t)$, can be written as $\bar x(t)=\Phi(\bar x_{(1)}(t),\ldots,\bar x_{(m)}(t);\lambda)$ for a generic family of particular solutions $\bar x_{(1)}(t),\ldots, \bar x_{(m)}(t)$ of $g_\di {\bf X}$ and { $\lambda\in N$}. Since $g^{-1}$ maps $g_\di {\bf X}$ onto ${\bf X}$, the general solution, $x_{(0)}(t)$, of the system ${\bf X}$ can be brought into the form $x_{(0)}(t)=g^{-1}_t(\bar x(t))$. Moreover, if $x_{(1)}(t),\ldots,x_{(m)}(t)$ are particular solutions of ${\bf X}$, then $g_t(x_{(i)}(t))$, with $i=1,\ldots,m$, are particular solutions to $ g_\di{\bf X}$. Hence, the general solution of ${\bf X}$ can be written as follows
\begin{equation}\label{e8}
x_{(0)}(t)=g_t^{-1}(\bar x(t))=g_t^{-1}\circ \Phi(g_t(x_{(1)}(t)),\ldots,g_t(x_{(m)}(t)),\lambda),\qquad \lambda\in N,
\end{equation}
in terms of a generic family $x_{(1)}(t),\ldots,x_{(m)}(t)$ of particular solutions to ${\bf X}$. Hence,  (\ref{eq9}) becomes a {\it $t$-dependent superposition rule} for ${\bf X}$.
\end{proof}
The $t$-dependent superposition rule (\ref{eq9}) for quasi-Lie systems is meaningful in obtaining the
general solution of a system of PDEs $\bX$ provided the generalised $t$-flow $g$ is explicitly known.
In practice, it is crucial to obtain a control $g$ from a
system of PDEs that can be integrated effectively \cite{CGL08}. To obtain it, we will extend the theory of quasi-Lie schemes to systems of PDEs. The notion of a quasi-Lie scheme remains as in the case of the study of ODEs \cite{CGL08}.


\begin{definition}\label{QLscheme}{\rm A {\it quasi-Lie scheme} is a pair $(W,V)$ of finite-dimensional real vector spaces of vector fields on $N$ satisfying that
$$
W\subset V,\qquad [W,W]\subset W,\qquad [W,V]\subset V.
$$

Meanwhile, the quasi-Lie scheme must be now attached to more general structures than in \cite{CGL08} to tackle the study of PDEs. This is accomplished in the following.

Consider the family $V(\mathbb{R}^s)$ of
$t$-dependent polyvector fields ${\bf X}=\sum_{\pi=1}^sX_\pi\otimes e^\pi$ on $N$ satisfying the ZCC condition and such that each $({X}_\pi)_t$, with $t\in\R^s$, takes values in $V$. The elements of $V(\mathbb{R}^s)$ can be parametrised by a family of functions on $\mathbb{R}^s$. More specifically, take a basis
 $\{Y_1,\dots,Y_r\}$ of $V$. Every ${\bf X}\in V(\mathbb{R}^s)$ amounts then to the unique family of functions
$\vec{\bf b}=\vec{\bf b}(t)=(\vec{b}_1(t),\dots,\vec{b}_r(t))$ such that
$$
(X_\pi^{b})_t=\sum_{j=1}^rb^\pi_j(t)Y_j\otimes e^\pi\,,\qquad \vec{b}_j(t)=(b_j^1(t),\ldots,b_j^s(t)),
$$
and we denote ${\bf X}$ by ${\bf X}^{\bf \vec b}$. The vector sign on ${\bf \vec b}$ will be dropped when describing time-dependent vector fields.

The $t$-dependent system of PDEs related to an $\bX\in V(\R^s)$ is
not a Lie system in general, if $V$ is not a Lie algebra itself. Meanwhile, the Lie--Scheffers Theorem for PDEs (see \cite{CGM07,OG00}) can be reformulated
as follows.

\begin{theorem} {\bf (PDE Lie--Scheffers Theorem)} A t-dependent system of PDEs ${\bf X}$ on $N$ is a PDE Lie system if and only
if there exists a finite-dimensional Lie algebra of vector fields $V_0$ on $N$ such that ${\bf X}\in V_0(\mathbb{R}^s)$.
\end{theorem}

Let us use the elements of $W$ to generate controls mapping the elements of $V(\mathbb{R}^s)$ into PDE Lie systems or any other simpler system of PDEs to be studied.

\begin{definition} {\rm The {\it group of the quasi-Lie scheme} $(W,V)$ is the group, $\mathcal{G}^s(W)$, spanned by generalised $t$-flows with arbitrary foot points corresponding to the integrable $t$-dependent polyvector fields taking values in $W$}.
\end{definition}

Let us describe our definition of the group of a quasi-Lie scheme more carefully since it has one difference with previous works, where only generalised $t$-flows with foot point $t_0=0$ are employed \cite{CGL08}. This subtle difference allows for the more accurate description of integrability conditions through quasi-Lie schemes.

The Lie algebra $W$ can be integrated to a Lie group action $\varphi:G\times N\rightarrow N$ whose fundamental vector fields are exactly the elements of $W$. Then, the generalised $t$-flows with an arbitrary foot point $t_0\in \mathbb{R}^s$ for the elements of $W(\mathbb{R}^s)$ are of the form $\varphi(\sigma(t),x)$, where $\sigma(t_0)=e$. Then, the elements of $\mathcal{G}^s(W)$ are $t$-dependent changes of variables of the form $\varphi(\sigma(t),x)$, where $\sigma(t)$ is a product of mappings from $\mathbb{R}^s$ to $G$ taking the value $e$. Observe that every mapping $\sigma:\mathbb{R}^s\rightarrow G$ can be written as a product of two mappings taking the value $e$. In fact, there always exists a mapping $\sigma_1(t)$  such that $\sigma_1(0)=e$ and $\sigma_1(t_1)=\sigma(t_1)$. Then $\sigma_2(t)=\sigma(t)\sigma_1^{-1}(t)$ takes the value $e$ at $t_1$ and $\sigma(t)=\sigma_1(t)\sigma_2(t)$ is the product of two mappings $\sigma_1(t), \sigma_2(t)$ taking the value $e$. Hence, the elements of $\mathcal{G}^s(W)$ are $t$-dependent changes of variables of the form $\varphi(\sigma(t),x)$ for an arbitrary $\sigma:\mathbb{R}^s\rightarrow G$. This amounts to the previously defined extended $t$-flows.

The usefulness of the group of a quasi-Lie scheme is based upon the proposition below.

\begin{proposition}\label{Main} {\bf (Main property of a quasi-Lie scheme)}
{\rm Given a quasi-Lie scheme $(W,V)$, a $t$-dependent polyvector field $\bX\in V(\mathbb{R}^s)$, and an extended $t$-flow $g\in
\mathcal{G}^s(W)$, we get that $g_\di\bX\in V(\mathbb{R}^s)$.}
\end{proposition}
The proof of Proposition \ref{Main} is an immediate generalisation of \cite[Proposition 1]{CGL08} and follows from the fact that if $g$ is a composition of generalised $t$-flows induced by  $t$-dependent polyvector fields of $W(\mathbb{R}^s)$. One has to consider the action of $g$ on the coordinates of ${\bf X}=(X_1,\ldots, X_s)$. The action of $g$ on each $X_i$ belongs to $V$ as a consequence of \cite[Proposition 1]{CGL08}. Hence, the $t$-dependent polyvector field $g_\di \bX$ takes values in $V$.

The coefficients of ${\bf Y}\in W(\mathbb{R}^s)$ in the basis $\{Y_1,\ldots, Y_r\}$ permit us to parametrise its extended $t$-flows in the form $g^{{\bf \sigma}}=\underline{g}({\bf \sigma}(t),g_0)$, where $\sigma(t)$ is a mapping from $\mathbb{R}^s$ to the a Lie group $G$ obtained by integrating $W$. Let us assume additionally that $V_0$ is a Lie algebra of vector fields contained in the linear space $V$. We can then look for $\sigma(t)$ such that for a certain $\vec {\bf b}(t)$ we get that
\begin{equation}\label{Inclusion}
g^{\bf \sigma}_\di \bX^{\bf \vec b}\in V_0(\mathbb{R}^s).
\end{equation}
This choice of control functions makes $(\bX^{\bf \vec b},g^{\bf \sigma})$ into a quasi-Lie system, which allows us to apply Theorem \ref{t4} to construct a $t$-dependent superposition rule for $\bX^{\bf \vec b}$.

Let us observe that the inclusion (\ref{Inclusion}) becomes a
differential equation for the ${\bf \sigma}(t)$ in terms of the functions $\vec{\bf b}(t)$. This situation is much more frequent in the literature as it may seem to be at first sight (cf. \cite{CGL08,CGL09,CLL08Emd,CL08Diss,CL11}).

}
\end{definition}


 If $g_\di{\bf X}$ takes values in a finite-dimensional Lie algebra $V_0\subset V$, one obtains interesting results (see \cite{CGL08}) motivating the following definition.
\begin{definition} {A {\it quasi-Lie system with respect to quasi-Lie scheme $(W,V)$} is a a $t$-dependent polyvector field ${\bf X}\in V(\mathbb{R}^s)$ for which exists an extended $t$-flow $g\in
\mathcal{G}^s(W)$ and a Lie algebra of vector fields $V_0\subset V$ such that
$
g_\di\bX\in V_0(\R^s).
$}
\end{definition}
It is worth noting that if ${\bf X}$ is a quasi-Lie system with respect to the quasi-Lie scheme $(W,V)$, then it automatically admits a
$t$-dependent superposition rule in the form given by (\ref{eq9}). In the following, we will apply our theory
and illustrate these concepts with examples.

From now on, given a quasi-Lie scheme $(W,V)$, a  $g\in\mathcal{G}^s(W)$, and a Lie algebra of
vector fields $V_0\subset V$, we denote by $(W,V;V_0)_g$ the set of quasi-Lie systems relative to
$(W,V)$ such that $g_\di\bX\in V_0(\mathbb{R}^s)$.

\begin{corollary}\label{QL}
The family of quasi-Lie systems $(W,V;V_0)_g$ on $N$ is a PDE Lie family admitting the common $t$-dependent
superposition rule of the form
\begin{equation}\label{sf}
\bar{\Phi}_g(t,x_{(1)},\dots,x_{(m)},\lambda )=g_t^{-1}\circ\Phi\left(g_t(x_{(1)}),\dots,g_t(x_{(m)}),\lambda \right)\,,
\end{equation}
for any $t$-independent superposition rule $\Phi$ of a PDE Lie system with a Vessiot--Guldberg Lie algebra $V_0$.
\end{corollary}

\begin{proof} As a consequence of Corollary \ref{CorFin}, every family of PDE Lie systems admitting a common Vessiot--Guldberg Lie algebra, $V_0$, admits a common superposition rule $\Phi$. The proof of the present corollary then follows from applying Theorem \ref{t4}.
\end{proof}

In view of Corollary \ref{QL}, every quasi-Lie system and, consequently, every PDE Lie system can be included
in a PDE Lie family satisfying Theorem \ref{MT}. This fact  justifies once more calling this theorem generalised Lie--Scheffers
Theorem.

{\begin{example} Let us consider the vector space  of vector fields on $\mathbb{R}$  given by $V_{\rm Ab}=\langle u^3\partial_u,u^2\partial_u,u\partial_u,\partial_u\rangle $
and the Lie algebra $W_{\rm Ab}=\langle u\partial_u,\partial_u\rangle\subset V$. It is immediate that $(W_{\rm Ab},V_{\rm Ab})$ becomes a quasi-Lie scheme. The $t$-dependent polyvector fields taking value in the space $V_{\rm Ab}(\mathbb{R}^2)$ correspond to the systems of Abel PDEs
\begin{equation}\label{PRE*}
\left\{
\begin{aligned}
\frac{\partial u}{\partial t_1}&=A(t_1,t_2)u^3+B(t_1,t_2)u^2+C(t_1,t_2)u+D(t_1,t_2)\,,\\
\frac{\partial u}{\partial t_2}&=E(t_1,t_2)u^3+F(t_1,t_2)u^2+G(t_1,t_2)u+H(t_1,t_2)\,,
\end{aligned}\right.
\end{equation}
where $A(t_1,t_2),B(t_1,t_2),C(t_1,t_2),D(t_1,t_2),E(t_1,t_2),F(t_1,t_2)$  satisfy the ZCC condition, namely
$$
AF-EB=0,\quad \partial_{t_2}A-\partial_{t_1}E+2(AG-CE)=0,\quad \partial_{t_2}C-\partial_{t_1}G+2(BH-DF)=0,
$$
$$
\partial_{t_2}B-\partial_{t_1}F+3(AH-DE)+BG-CF=0,\qquad \partial_{t_2}D-\partial_{t_1}H+CH-DG=0.
$$
The group $\mathcal{G}^2(W_{\rm Ab})$ can be described by curves in the group $\mathfrak{Aff}(\mathbb{R})\simeq \mathbb{R}_*\ltimes \mathbb{R}$ of affine transformations on $\mathbb{R}$. Then, $\mathcal{G}^2(W_{\rm Ab})$ consists of generalised $t$-flows of the form
$$g(t_1,t_2)(x)=a(t_1,t_2)u+b(t_1,t_2),$$
where $a(t_1,t_2)$ is a positive function and $b(t_1,t_2)$ is arbitrary. These non-autonomous changes of variables act on the equations (\ref{PRE*}). Let $V_0=\langle u^3\partial_u,u\partial_u\rangle$. The space $V_0(\mathbb{R}^2)$ contains the $t$-dependent polyvector fields corresponding to the systems of PDEs
$$
 \left\{
 \begin{aligned}
 \frac{\partial u}{\partial t_1}&=f(t_1,t_2)u^3+l(t_1,t_2)u\,,\\
 \frac{\partial u}{\partial t_2}&=h(t_1,t_2)u^3+m(t_1,t_2)u\,,
 \end{aligned}\right.
$$
for functions $f(t_1,t_2),h(t_1,t_2),l(t_1,t_2),m(t_1,t_2)$ satisfying the ZCC condition.
If the $t$-dependent polyvector field $\bX$ corresponds to (\ref{PRE*}) and $g\in\mathcal{G}^s(W_{\rm Ab})$, then $g_\di X$ takes the form
{\footnotesize
\begin{equation*}\label{PRE2*}
\left\{
\begin{aligned}
\frac{\partial g}{\partial t_1}&=\frac{A}{a^2}g^3+\frac{aB-3Ab}{a^2}g^2+\frac{3Ab^2/a-2bB+Ca+\partial_{t_1}a}{a}g+\frac{b(bB-Ab^2/a-Ca-\partial_{t_1}a)+a^2D+a\partial_1b}{a}\,,\\
\frac{\partial g}{\partial t_2}&=\frac{E}{a^2}g^3+\frac{aF-3Eb}{a^2}g^2+\frac{3Eb^2/a-2bF+Ga+\partial_{t_2}a}{a}g+\frac{b(bF-Eb^2/a-Ga-\partial_{t_1}a)+a^2H+a\partial_1b}{a}\,.
\end{aligned}\right.
\end{equation*}
}
Recalling the ZCC condition for (\ref{PRE*}), we get that $g_\di \bX\in V_0(\R^2)$ if and only if $aB-3Ab=0$ and, the conditions
\begin{equation}\label{conAb}
\begin{gathered}
27DE^2-9FEC+2BF^2-9F\partial_1E+9E\partial_1F=0,\\
2HDE^2-9FEG+2F^3-9F\partial_2E+9E\partial_2F=0.
\end{gathered}
\end{equation}
are satisfied. Then, the initial PDE Abel differential equation can be mapped to a PDE Bernoulli equation $g_\di \bX$, which
admits a $t$-dependent superposition rule constructed from Theorem \ref{t4} of the form
$$
{u=\Phi(u_{(1)},u_{(2)},\lambda )=[\lambda u^{-1/2}_{(1)}+(1-k)u^{-1/2}_{(2)}]^{-2}},\quad \lambda\in \mathbb{R}.
$$
By inverting the $t$-dependent change of variables, one gets a $t$-dependent superposition rule for ${\bf X}$. All partial Abel differential equations (\ref{PRE*}) satisfying the condition (\ref{conAb}) form a PDE Lie family.
\end{example}

\section{Quasi-Lie invariants and integrability}

Let us restudy geometrically the generalised Chiellini conditions appearing in Example \ref{e1a} to show an interesting fact that was overseen in \cite{HLM13}. This will lead to develop techniques for determining the control of a quasi-Lie system by using a certain type of invariants related to quasi-Lie schemes.

The form of a generalised Abel differential equation (\ref{InAbe}), for a fixed $\epsilon\in \mathbb{R}$, can be characterised by its $t$-dependent coefficients ${\bf b}={\bf b}(t)=(a(t),c(t),f(t),g(t))$. This can be restated geometrically by saying that a generalised Abel differential equation (\ref{InAbe}) amounts to a unique $t$-dependent vector field $X$ taking values in the vector space $V_{\rm GA}$ spanned by the basis of vector fields
\begin{equation}\label{Bas}
X_1=\partial_x,\qquad X_2=x\partial_x,\qquad X_3=x^{\epsilon-1}\partial_x,\qquad X_4=x^\epsilon\partial_x
\end{equation}
and $X=a(t)X_1+c(t)X_2+f(t)X_3+g(t)X_4$ for a unique set ${\bf b}$ of $t$-dependent coefficients. We will hereafter write $X^{\bf b}$ for the $t$-dependent vector field $X^{\bf b}=a(t)X_1+c(t)X_2+f(t)X_3+g(t)X_4$ associated with a certain ${\bf b}$.

As $X^{\bf b}$ can be considered as a curve ${\bf b}(t)$ in $V_{\rm GA}$, the jets ${\bf j}^1_tX$ can be understood as a curve ${\bf j}^1X^{\bf b}(t)$ in $T^1V_{\rm GA}$.  If $\{x_1,x_2,x_3,x_4\}$ are the coordinates on $V_{\rm GA}$ corresponding to the basis (\ref{Bas}), then $\{x_1,x_2,x_3,x_4,\dot x_1,\dot x_2,\dot x_3,\dot x_4\}$ stands for the induced coordinate system in $T^1V_{\rm GA}$ and a general point of $T^1V_{\rm GA}$ can be represented by $(x_1,x_2,x_3,x_4,\dot x_1,\dot x_2,\dot x_3,\dot x_4)$.
 Hence, in coordinates,
$$
{\bf j}^1X^{\bf b}(t)=(a(t),c(t),f(t),g(t),\dot a(t),\dot c(t),\dot f(t),\dot g(t)).
$$

Consider the $t$-dependent vector fields  ${X}^{\bf b}$ and ${X}^{\bf  b'}$ associated with the initial (\ref{InAbe}) and the simplified (\ref{OutAbe}) generalised Abel differential equations. Recall that $X^{\bf b'}$ was chosen in such a way that we can obtain its general solution. 

We now propose that the generalised Chiellini conditions (\ref{GCC}) must be understood as two functions $F_i:T^1V_{\rm GA}\rightarrow \mathbb{R}$, with $i=1,2$, namely
\begin{equation}\label{AbelInv}
F_1=-\frac{x_4^{\epsilon-3}}{x_3^\epsilon}(x_4\dot x_3-(x_2x_4+\dot x_4)x_3),\qquad F_2=\frac{x_4^{\epsilon-1}x_1}{x_3^\epsilon},
\end{equation}
such that $F_1({\bf j}_t^1X^{\bf b})=k_1$ and $F_2({\bf j}_t^1X^{\bf b'})=k_2$ for every $t\in \mathbb{R}$. In other words, $F_1$ and $F_2$ take a constant value on ${\bf j}_t^1X^{\bf b}$. To understand a little bit more why the generalised Chiellini conditions allow us to integrate $X^{\bf b}$, one may notice that if $X^{\bf b}$ can be connected with $X^{\bf b'}$ through a $t$-dependent transformation $\bar x=\beta(t)x$, where $\beta(t)$ is any positive function, then $X^{\bf b'}$ is related to the generalised Abel differential equation
\begin{equation}\label{Ex1}
\frac{d\bar x}{dt}=a(t)\beta(t)+\left(\frac{\dot\beta(t)}{\beta(t)}+b(t)\right)\bar x+\frac{f^{\epsilon-2}(t)}{\beta(t)}\bar x^{\epsilon-1}+\frac{g(t)}{\beta^{\epsilon-1}(t)}\bar x^\epsilon,
\end{equation}
and, omitting the dependence on $t$ of the coefficients for simplicity,
\begin{equation}\label{Ex2}
\!\!{\bf j}^1\!X^{\bf b'}\!\!=\!\!\left(a\beta,\frac{\dot\beta}{\beta}+b,\frac{f^{\epsilon-2}}{\beta},\frac{g}{\beta^{\epsilon-1}},\dot a\beta\!+\!a\dot\beta,\frac{\ddot\beta}{\beta}\!-\!\frac{\dot\beta^2}{\beta^2}+\dot b,\frac{\dot f\beta\!+\!(2\!-\!\epsilon)f\dot\beta}{\beta^{\epsilon-1}},\frac{\dot g \beta+(1-\epsilon)g\dot \beta}{\beta^{\epsilon}}\right).
\end{equation}
Hence, $F_i({\bf j}_t^1X^{\bf b'})=F_i({\bf j}_t^1X^{\bf b})$ for every $t\in \mathbb{R}$. Since we aim to map $X^{ \bf b}$ onto $X^{\bf  b'}$ through such a $t$-dependent transformation and $F_1,F_2$ are constant on ${\bf j}^1X^{ \bf b'}(t)$, because it takes values in a one-dimensional space of $V_{\rm GA}$, then $F_1,F_2$ must be constant on the curve ${\bf j}^1X^{\bf b}(t)$.

Let us provide a theoretical framework for the above example that will be extensible to the integration of PDEs through quasi-Lie schemes extending and simplifying the approach given in \cite{CL15}.
\begin{definition}\label{Def:QLI}
A {\it quasi-Lie invariant} of first order relative to a quasi-Lie scheme $(W,V)$ on a manifold $N$ is a function $F:J^1_0(\mathbb{R}^s,V\otimes \mathbb{R}^s)\rightarrow \mathbb{R}$, where $J^1_0(\mathbb{R}^s,V\otimes \mathbb{R}^s)$ is {the} space of 1-jets at 0 of maps from $\mathbb{R}^s$ to $V \otimes \mathbb{R}^s$, such that
\begin{align*}
F({\bf j}^1_tg_\di {\bf X})=F({\bf j}^1_t{\bf X}),\qquad \forall g\in \mathcal{G}^s(W),\qquad \forall {\bf X}\in V(\mathbb{R}^s),\qquad \forall t\in\mathbb{R}.
\end{align*}
\end{definition}

Although we focus on the case of quasi-Lie invariants of first-order, which is enough to analyse the examples treated in this section, the theory for quasi-Lie invariants of higher order follows straightforwardly from our results.

It is worth noting that the work \cite{CL15} presents a different definition of quasi-Lie invariants for ODEs as a certain function $F:V(\mathbb{R})\rightarrow C^\infty(\mathbb{R})$ such that $F(g_\di X)=F(X)$ for every $g\in \mathcal{G}^1(W)$ and $X\in V(\mathbb{R})$. Nevertheless, all examples treated in \cite{CL15} through the latter definition of quasi-Lie invariants are, at the end, reduced to study our Definition \ref{Def:QLI} (cf. \cite[Sec. 10, 11, and 12]{CL15}), which is purely geometrical and it is simpler in the sense that it does not rely on the use of infinite-dimensional spaces like $V(\mathbb{R})$.


\begin{example} Let us define the quasi-Lie scheme $(W_{\rm GA},V_{\rm GA})$, where $W_{\rm GA}=\langle x\partial_x\rangle$. The integration of the vector field $x\partial_x$ gives rise to a Lie group action $\phi:(\lambda,x)\in \mathbb{R}_*\ltimes\mathbb{R}\mapsto \lambda x\in \mathbb{R}$. Meanwhile, $\mathcal{G}^1(W_{\rm GA})$ is given by the $t$-changes of variables of the quasi-Lie scheme, namely $\bar x(t)=\beta(t)x(t)$ for a positive function $\beta(t)$. The expression (\ref{Ex2}) stands for $g_\di X^{\bf b}$. In view of Definition \ref{Def:QLI} and the expression (\ref{Ex2}), one sees that the functions $F_1,F_2:J^1_0(\mathbb{R},V_{\rm AG})=T^1V_{\rm GA}\rightarrow \mathbb{R}$ are  quasi-Lie invariants of first-order relative to $(W_{\rm GA},V_{\rm GA})$.
\end{example}

To determine all quasi-Lie invariants of first order, we will use the fact that if $G$ is a Lie group with multiplication $'\cdot'$, then the space of $2$-order jets of maps from $\mathbb{R}^s$ to the Lie group $G$  at $0\in \mathbb{R}^s$, let us say $J^2_0(\mathbb{R}^s,G)$, is a differentiable manifold admitting a Lie group structure relative to the multiplication
$$
\begin{array}{rccc}
\star:&J_0^2(\mathbb{R}^s,G)\times J_0^2(\mathbb{R}^s,G)&\longrightarrow &J_0^2(\mathbb{R}^s,G)\\
&({\bf j}^2_0g_1, {\bf j}_0^2g_2)&\mapsto &{\bf j}_0^2(\hat g_1\cdot \hat g_2),
\end{array}
$$
where $\hat g_1,\hat g_2$ are maps from $\mathbb{R}^s$ to $G$ belonging to the equivalence classes of ${\bf j}^2_0g_1$ and ${\bf j}^2_0g_2$, respectively. This result is independent of the representatives $\hat g_1,\hat g_2$ since ${\bf j}^2_0(\hat g_1\cdot \hat g_2)$ depends only on the derivatives up to second order of $\hat g_1,\hat g_2$. The associativity of $\star$ follows from the previous commentary and the associativity of the multiplication in $G$.

\begin{example}\label{JetGroup} Consider the Lie group structure on $\mathbb{R}_*$ relative to its multiplicative group structure. Then,   
the Lie group $J^2_0(\mathbb{R},\mathbb{R}_*)$ has a group multiplication
$
(\lambda ,\dot\lambda,\ddot \lambda)\star (\mu,\dot \mu,\ddot \mu)=(\lambda\mu ,\mu\dot\lambda +\lambda \dot\mu,\ddot \lambda\mu+2\dot\lambda\dot\mu+\lambda\ddot \mu)
$
and its left-invariant vector fields are linear combinations of the vector fields
$
X_1=\lambda\partial_\lambda+\dot\lambda\partial_{\dot\lambda}+\ddot \lambda\partial_{\ddot\lambda}, X_2=\lambda \partial_{\dot\lambda}+2\dot\lambda\partial_{\ddot\lambda}, X_3=\lambda\partial_{\ddot\lambda}.
$
\end{example}

The following proposition attaches a quasi-Lie scheme to a Lie group action whose invariants, in a sense given next, give rise to quasi-Lie invariants of first order. It is worth recalling that every ${\bf X}\in V(\mathbb{R}^s)$ can be written as $(X_1,\ldots,X_s)$, where $X_1,\ldots, X_s$ are functions from $\mathbb{R}^s$ to $V$. In consequence, every ${\bf X}\in V(\mathbb{R}^s)$ can be considered as a mapping from $\mathbb{R}^s$ to $V\otimes \mathbb{R}^s$. 

\begin{proposition}\label{LGAP}
Let $(W,V)$ be a quasi-Lie scheme on $N$ and let $\varphi:G\times N\rightarrow N$ be the Lie group action induced by the integration of the Lie algebra of vector fields $W$. Then, there exists a Lie group action of the form
\begin{equation}\label{LGA}
\begin{array}{rccc}
\varphi^1&:J_0^{2}(\mathbb{R}^s,G)\times J_0^{1}(\mathbb{R}^s,V\otimes \mathbb{R}^s)&\longrightarrow &J_0^{1}(\mathbb{R}^s,V\otimes \mathbb{R}^s),\\
&({\bf j}^{2}_0g,{\bf j}^{1}_0{\bf X})&\mapsto &{\bf j}_0^{1}(\hat g_\di {\bf \hat X}),
\end{array}
\end{equation}
where the space $J^{1}_0(\mathbb{R}^s,V\otimes \mathbb{R}^s)$ is the bundle of $1$-jets at $0$ of mappings from $\mathbb{R}^s$ to $V\otimes \mathbb{R}^s$, and $\hat g, \hat X$ are representatives of ${\bf j}_0^{2}g$ and ${\bf j}_0^{1}{\bf X}$, respectively.
\end{proposition}
\begin{proof} Let us show that $\varphi^1$ takes values in $J^1_0(\mathbb{R}^s,V\otimes \mathbb{R}^s)$.
Proposition \ref{Main} ensures that $\hat g_\di {\bf \hat X}\in V(\mathbb{R}^s)$, i.e. $\hat g_\di {\bf \hat X}$ can be understood as map from $\mathbb{R}^s$ to $V\otimes \mathbb{R}^s$. In consequence, the $1$-order jet at $0\in \mathbb{R}^s$ of $\hat g_\di {\bf \hat X}$ can be understood as an element of $J^1_0(\mathbb{R}^s,V\otimes \mathbb{R}^s)$.

Let us prove now that ${\bf j}_0^{1}(\hat g_\di {\bf \hat X})$ is a function depending only on  ${\bf j}_0^{2}g$ and ${\bf j}_0^{1}{\bf X}$. If $x(t)$ is a particular solution to $\hat {\bf X}$ with initial condition $x(0)=\hat g^{-1}_0(x)$ and using that every $\hat g\in \mathcal{G}^s(W)$ can be written as $\hat g(t,x)=\varphi(\sigma(t),x)$ for a certain mapping $\sigma:\mathbb{R}^s\rightarrow G$, then $\hat g(t,x(t))$ is a particular solution of the $t$-dependent polyvector field $\hat g_\di \hat X$ and
$$
\begin{aligned}
(\hat g_\di {\bf \hat X})_0(x)&=\sum_{\pi=1}^s\frac{\partial}{\partial t_\pi}\bigg|_{t=0}\varphi(\sigma(t),x(t))\otimes e^\pi\\
&=\sum_{\pi=1}^s\frac{\partial}{\partial t_\pi}\bigg|_{t=0}\left[\varphi(\sigma(t)\sigma^{-1}(0),x)+\varphi(\sigma(0),x(t))\right]\otimes e^\pi\\
&={\bf Y}^{\hat{\sigma}}(x)+\sum_{\pi=1}^s\frac{\partial}{\partial t_\pi}\bigg|_{t=0}\left[\varphi(\sigma(0),x(t))\right]\otimes e^\pi,
\end{aligned}
$$
where ${\bf Y}^{\sigma}$ is the $t$-dependent polyvector field taking values in $W(\mathbb{R}^s)$ whose generalised $t$-flow is determined by $\hat \sigma(t)=\sigma(t)\sigma^{-1}(0)$. Therefore, ${\bf Y}^{\hat\sigma}$ is a function only of ${\bf j}^1_0g$. Meanwhile,  Proposition \ref{Main} ensures that $\partial_{t_i}[\varphi(g(0),x(t))]$ gives rise to an element of $V(\mathbb{R}^s)$ and its value at $t=0$ depends on ${\bf j}^0_0g$ and ${\bf j}^0_0{\bf X}$ (see also \cite{CGL08,CLL08Emd}). Deriving the expression for $\hat g_\di \hat X$ in terms of the $t_\pi$, with $\pi=1,\ldots,s$, we find that ${\bf j}^1_0\hat g_\di {\bf \hat X}$ depends only on ${\bf j}^{2}_0g$ and ${\bf j}^1_0{\bf X}$. In consequence,  $\varphi^1$ is a well-defined mapping.

Finally, let us prove that $\varphi^1$ is a Lie group action. In fact
\begin{multline*}
\varphi^1({\bf j}^{2}_0g_1,\varphi^1({\bf j}^{2}_0g_2,{\bf j}^{1}_0{\bf X}))=\varphi^1({\bf j}^{2}_0g_1,{\bf j}_0^{1}({\hat{g}_2}\,_\di \hat {\bf X}))={\bf j}^{1}_0({\hat{g}_1}\,_\di {\hat{g} _2}\,_\di\hat {\bf X})\\
={\bf j}^{1}_0((\hat{g}_1\cdot \hat{g}_2) _\di\hat {\bf X})=\varphi^1({\bf j}_0^{2}(\hat g_1\cdot \hat g_2),{\bf j}_0^{1}{\bf X})=\varphi^1({\bf j}_0^{2}g_1\star {\bf j}_0^{2}g_2,{\bf j}_0^{1}{\bf X})
\end{multline*}
and $\varphi^1({\bf j}_0^{2}e,{\bf j}_0^{1}{\bf X})={\bf j}_0^{1}{\bf X}$, where $e$ is understood as the mapping $e:t\in \mathbb{R}^s\mapsto e\in G$.
\end{proof}
The following is an immediate consequence of Proposition \ref{LGAP} and Definition \ref{Def:QLI}.

\begin{corollary}\label{QLI} Let $(W,V)$ be a quasi-Lie scheme and let $\varphi:G\times N\rightarrow N$ be the Lie action resulting from the integration of $W$. Every constant function $F:J^1_0(\mathbb{R}^s,V\otimes \mathbb{R}^s)\rightarrow \mathbb{R}$ along the orbits of the Lie group action $\varphi^1$ is a quasi-Lie invariant of order one relative to $(W,V)$.
\end{corollary}

Corollary \ref{QLI} retrieves as a particular case Corollary 12.2 and particular cases of Theorem 3.10 in \cite{CL15}. Relevantly, the generalisations of our results to quasi-Lie invariants of any order are immediate. In this sense, this improves all results in \cite{CL15}, where only a few necessary conditions for the existence of quasi-Lie invariants where given.

Let us now give the following application of Corollary \ref{QLI} to the description of the integrability of generalised Abel differential equations through the generalised Chiellini integrability conditions.
This will allow for the determination of $F_1$ and $F_2$ given in (\ref{AbelInv}) and in \cite{HLM13}.

\begin{example} We already know that $(W_{\rm GA},V_{\rm GA})$ is a quasi-Lie scheme and we described the elements of $\mathcal{G}^1(W_{\rm GA})$ through  the Lie group action $\varphi$. 

The functions $F_1$ and $F_2$ can be retrieved as follows. In view of the expressions (\ref{Ex2}), the Lie group $J^2_0(\mathbb{R},\mathbb{R}_*)$ induces a Lie group {action} $\phi^1:J^2_0(\mathbb{R},\mathbb{R}_*)\times J^1_0(\mathbb{R},V_{\rm GA})\rightarrow J^1_0(\mathbb{R},V_{\rm GA})$, given by
$$
\begin{array}{rccc}
\phi^1(\beta,\dot\beta,\ddot \beta; {\bf j}_0^1X)=\left(a\beta,\frac{\dot \beta}{\beta}+b,\frac{f}{\beta^{\epsilon-2}},\frac{g}{\beta^{\epsilon-1}},\dot a\beta+a\dot\beta,\frac{\ddot\beta\beta-\dot\beta^2}{\beta^2}+\dot b,\frac{\dot f\beta+(2-\epsilon)f\dot\beta}{\beta^{\epsilon-1}},\frac{\dot g \beta+(1-\epsilon)g\dot \beta}{\beta^{\epsilon}}\right).
\end{array}
$$
By using the results of Example \ref{JetGroup}, we obtain that the fundamental vector fields of the Lie group action $\phi^1$ are spanned by
$$
X_1^{[1]}=a\partial_a+(2-\epsilon)f\partial_f+(1-\epsilon)g\partial_g+\dot a\partial_{\dot a}+\dot b\partial_{\dot b}+(2-\epsilon)\dot f\partial_{\dot f}+(1-\epsilon)\dot g\partial_{\dot g}.
$$
$$
X_2^{[1]}=\partial_b+a\partial_{\dot a}+(2-\epsilon)f\partial_{\dot f}+(1-\epsilon) g\partial_{\dot g},\qquad X_3^{[1]}=\partial_{\dot b}.
$$
A straightforward computation shows that the functions $F_1$ and $F_2$ are constants of the motion of latter vector fields. In view of Corollary \ref{QLI}, both are quasi-Lie invariants of order one. Since the vector fields $X_1^{[1]},X_2^{[1]},X_3^{[1]}$ span a distribution of rank three on an eight-dimensional manifold, there exist four more functionally independent invariants of the vector fields $X_1^{[1]},X_2^{[1]},X_3^{[1]}$. For instance, one of them is the function $F_3:J^1_0(\mathbb{R},V_{\rm GA})\rightarrow\mathbb{R}$ of the form
$$
F_3({\bf j}^1_tX^{\bf b})=\frac{(\epsilon-1)g^{\epsilon-2}\dot gc+g^{\epsilon-1}\dot b-\epsilon \dot f g^{\epsilon-1}c}{f^{\epsilon+1}}.
$$
\end{example}

Note that if ${\bf X}^{\bf b}$ can be mapped into a Lie system related to a one-dimensional Lie algebra, then the {values of $F_1$ and $F_2$} must be constant on the final system. This provides a necessary condition for the integrability of ${\bf X}^{{\bf b}}$ through quasi-Lie schemes. Nevertheless, $F_3$ is not constant in quasi-Lie schemes related to $(W_{\rm GA},V_{\rm GA})$. Hence, quasi-Lie invariant can be used to obtain integrability conditions, but every {case} must be studied separately since only certain quasi-Lie invariants will be meaningful. The determination of sufficient conditions to ensure integrability and other related topics will be the topic of future works. The above procedure can be applied in a similar way to other works accomplished by the authors of \cite{HLM13}.

\section{Conclusions and Outlook}

We have characterised families of integrable PDEs in normal form admitting a common $t$-dependent superposition rule and we have established their relations to the theory of quasi-Lie schemes and quasi-Lie systems.  In the meanwhile, we have found new examples of PDE Lie systems occurring in B\"acklund transformations and the theory of Lax pairs \cite{Ra89,ZT09}.  This motivates the further analysis of the applications of PDE Lie systems, whose analysis has been mainly theoretical until now \cite{CGM07,Dissertationes,GR95,OG00}. The theory of quasi-Lie invariants has been improved by giving a more geometrical approach than in previous works \cite{CL15}, which has explained more concisely some of the methods for finding these invariants.

In the future, we aim to study generalisations of superposition rules to other types of differential equations, e.g. fractional differential and difference equations. Moreover, we plan to describe new applications of new and known types of superposition rules in mathematics and physics. Additionally, we aim to further investigate the use of quasi-Lie schemes in the integrability of PDEs.

\color{black}

\section*{Acknowledgements}
 Partial financial support by research projects {MTM2015-64166-C2-1-P} and E24/1 (DGA)
 are acknowledged. J. Grabowski and J. de Lucas acknowledge funding
from the Polish National Science Centre under grant HARMONIA 2016/22/M/ST1/00542.

\end{document}